\documentclass[10pt]{amsart}
\usepackage{latexsym}
\usepackage{amsmath}
\usepackage{amssymb}
\usepackage{mathrsfs}
\usepackage{graphicx}
\usepackage{color}
\usepackage{pgfpages}
\usepackage{ifthen}
\usepackage{leftidx,tensor}
\usepackage[T1]{fontenc}
\usepackage[latin1]{inputenc}
\usepackage{mathtools}
\usepackage{comment}
\usepackage{dsfont}
\usepackage[nocompress]{cite}

\usepackage[shortlabels]{enumitem}
\usepackage{aliascnt}
\usepackage[bookmarks=true,pdfstartview=FitH, pdfborder={0 0 0}, colorlinks=true,citecolor=red, linkcolor=blue]{hyperref}
\usepackage{bbm}
\usepackage{nicefrac}
\usepackage{tikz}
\usetikzlibrary{arrows.meta,bending}
\usepackage{float}

\theoremstyle{plain}
\newtheorem{thm}{Theorem}[section]
\newaliascnt{cor}{thm}
\newaliascnt{prop}{thm}
\newaliascnt{lem}{thm}
\newtheorem{cor}[cor]{Corollary}
\newtheorem{prop}[prop]{Proposition}
\newtheorem{lem}[lem]{Lemma}
\aliascntresetthe{cor}
\aliascntresetthe{prop}
\aliascntresetthe{lem} 
\theoremstyle{definition}
\newaliascnt{defn}{thm}
\newaliascnt{asu}{thm}
\newaliascnt{con}{thm}
\aliascntresetthe{defn}
\aliascntresetthe{asu}
\aliascntresetthe{con}
\newcounter{stp}
\newcounter{stpi}
\newcounter{stpci}
\newcounter{stpiii}

\theoremstyle{definition}
\newaliascnt{rem}{thm}
\newaliascnt{exa}{thm}
\newaliascnt{masu}{thm}
\newaliascnt{nota}{thm}
\newaliascnt{sett}{thm}
\newtheorem{rem}[rem]{Remark}

\aliascntresetthe{rem}
\aliascntresetthe{exa}
\aliascntresetthe{masu}
\aliascntresetthe{nota}
\aliascntresetthe{sett}
\numberwithin{equation}{section}

\setlist[enumerate]{font = \normalfont}

\renewcommand{\d}{\mathrm{d}}
\newcommand{\R}{\mathbb{R}}
\newcommand{\C}{\mathbb{C}}

\DeclareMathOperator{\re}{Re}

\DeclareMathOperator{\Id}{Id}

\DeclareMathOperator{\diag}{diag}

\newcommand{\tand}{\enspace \text{and} \enspace}

\newcommand{\A}{\mathrm{A}}

\newcommand{\bS}{\bold{S}}
\newcommand{\bB}{\bold{B}}
\newcommand{\bD}{\bold{D}}
\newcommand{\bQ}{\bold{Q}}
\newcommand{\bA}{\bold{A}}
\newcommand{\bF}{\bold{F}}
\newcommand{\bw}{\bold{w}}
\newcommand{\bAn}{\bA_0}
\newcommand{\bBn}{\bB_0}
\newcommand{\bDn}{\bD_0}

\newcommand{\Lm}{\L_{\mathrm{m}}}
\newcommand{\Gm}{\mathrm{G}_{\mathrm{m}}}
\newcommand{\rL}{\mathrm{L}}
\newcommand{\rH}{\mathrm{H}}
\newcommand{\rB}{\mathrm{B}}
\newcommand{\rW}{\mathrm{W}}

\newcommand{\mc}{\mathcal}
\renewcommand{\L}{\mathcal{L}}
\newcommand{\Deltam}{\Delta_{\mathrm{m}}}
\newcommand{\Deltan}{\Delta_0}
\renewcommand{\div}{\mathrm{div} \, }
\newcommand{\Am}{\A_{\mathrm{m}}}

\newcommand{\bBm}{\bB_{\mathrm{m}}}
\newcommand{\bDm}{\bD_{\mathrm{m}}}
\newcommand{\D}{\operatorname{div}}
\newcommand{\rC}{\mathrm{C}}

\renewcommand{\leq}{\leqslant}
\renewcommand{\geq}{\geqslant}

\newcommand{\An}{\A_0}
\newcommand{\Ln}{\L_0}
\newcommand{\G}{\mathrm{G}}
\newcommand{\Gn}{\G_0}

\newcommand{\E}{\mathbb{E}}

\newcommand{\sigmapdelta}{\sigma_{\mathrm{p}}^{\delta}}

\newcommand{\sigmaf}{\sigma_{\mathrm{f}}}
\newcommand{\up}{u_{\mathrm{p}}}
\newcommand{\pp}{p_{\mathrm{p}}}
\newcommand{\vp}{v_{\mathrm{p}}}
\newcommand{\uf}{u_{\mathrm{f}}}
\newcommand{\pf}{\pi_{\mathrm{f}}}
\newcommand{\mup}{\mu_{\mathrm{p}}}
\newcommand{\lambdap}{\lambda_{\mathrm{p}}}
\newcommand{\Omegap}{\Omega_{\mathrm{p}}}
\newcommand{\Omegaf}{\Omega_{\mathrm{f}}}
\newcommand{\Gammap}{\Gamma_{\mathrm{p}}}
\newcommand{\Gammaf}{\Gamma_{\mathrm{f}}}
\newcommand{\bX}{\bold{X}}
\newcommand{\Xn}{\bX_0}
\newcommand{\Xe}{\bX_1}

\newcommand{\bY}{\bold{Y}}
\newcommand{\Yn}{\bY_0}
\newcommand{\Ye}{\bY_1}

\newcommand{\bZ}{\bold{Z}}
\newcommand{\Zn}{\bZ_0}
\newcommand{\Ze}{\bZ_1}

\begin{document}

	\title{Strong well-posedness of a fluid--poro\-viscoelastic interaction problem: An approach by Spectral analysis}
	
	
	
	
	\author{Tim Binz} 
	\address{T. Binz, 
		Program in Applied \& Computational Mathematics (PACM)
		Princeton University, 
		Fine Hall, Washington Road, 08544 Princeton, NJ, USA.}
	\email {tb7523@princeton.edu}
	
	\author{Matthias Hieber} 
	\address{M. Hieber, 
		Fachbereich Mathematik,
		Technische Universit\"{a}t Darmstadt,
		Schlo\ss{}gartenstra{\ss}e 7, 64289 Darmstadt, Germany.}
	\email {hieber@mathematik.tu-darmstadt.de}

	\author{Arnab Roy}
	\address{A. Roy, Basque Center for Applied Mathematics (BCAM), Alameda de Mazarredo 14, 48009 Bilbao, Spain.}
	\address{IKERBASQUE, Basque Foundation for Science, Plaza Euskadi 5, 48009 Bilbao, Bizkaia, Spain.}
	\email{aroy@bcamath.org}

	

	\begin{abstract}
This article investigates a coupled viscoelastic Navier--Stokes--Biot system describing
the interaction between an incompressible viscous fluid and a poro--viscoelastic
medium in three spatial dimensions. The coupling between the fluid and the porous medium is realized through
Beavers--Joseph--Saffman type interface conditions. Using spectral analysis, it is proved that the coupled system admits a unique, strong, global  solution for small initial data. In addition, a Serrin--type blow-up criterion is established. 

	\end{abstract}
	
	\maketitle
	
\section{Introduction}


In this work, we investigate a three-dimensional fluid--structure interaction problem describing the coupling between an incompressible viscous fluid and a poro--viscoelastic medium. The fluid motion is governed by the Navier--Stokes equations, while the porous structure is modeled by the Biot system. In the two-dimensional setting, the existence of weak solutions for a related coupled system with a moving interface governed by a reticular plate equation was established in \cite{KCM:24}. More recently, a fully averaged poroelastic Kirchhoff plate model coupled with the linear Stokes equations was studied in \cite{brandt2026}, where the authors proved the existence of weak solutions as well as the existence and uniqueness of strong solutions to a regularized problem with the structure located on the boundary.

The main objective of the present work is to establish the {\em existence and uniqueness of strong solutions} for a three-dimensional fluid--structure interaction system coupling the Navier--Stokes equations with a general poro--viscoelastic Biot model. In contrast to the aforementioned works, the porous medium is fully surrounded by the fluid, and the coupling is imposed across fixed interfaces. Under a suitable smallness assumption on the initial data, we prove the global-in-time existence and uniqueness of strong solutions. 

A widely used description of the porous solid is provided by the Biot model \cite{Bio:41}, which couples the fluid flow in a saturated porous medium
with the elastic response of the solid skeleton and has been extensively studied analytically \cite{JLA:80, Sho:00}.
Motivated both by physical considerations and mathematical structure, particularly in applications involving biological tissues, it is natural to incorporate viscoelastic effects of Kelvin-Voigt type into the elastic stress. Such poro-viscoelastic models already appear in
Biot's work \cite{Bio:56} and have been further developed in the literature \cite{BCMW:21,BGSW:16,BMW:23}. Note that  the total stress tensor in the constitutive law for the porous solid is given by
\begin{equation*}
		\sigmapdelta(\up,\pp)=\lambdap \varepsilon\left(\up+\delta \cdot \partial_t \up\right) + 2\mup\varepsilon\left(\up+\delta \cdot\partial_t \up\right)-\alpha \pp\mathbb{I},
	\end{equation*}
where $\varepsilon(u)=\tfrac12(\nabla u+\nabla u^{\top})$ denotes the linearized strain
tensor and $\lambdap,\mup>0$ are the Lam\'e coefficients of elasticity, see \cite{BGSW:16}. The parameter $\delta\ge 0$ plays a central role in the model.  When $\delta=0$, the stress depends only on the instantaneous strain and
the system reduces to the classical dynamic \emph{poroelastic} Biot system. In contrast, for $\delta>0$ the stress additionally depends on the strain rate. The presence of $\delta>0$ has a profound impact on the spectral
properties of the associated linearized operator.

The coupling between the fluid and the porous solid is imposed on the interface
 via transmission conditions
expressing continuity of the normal fluid flux, balance of normal stresses, and tangential slip of
Beavers-Joseph-Saffman type (see Section \ref{sec2} for details).

The mathematical analysis of fluid--poroelastic interaction problems has so far been largely confined to weak solution theories. The author in
\cite{Sho:05} established the existence of weak solutions for the poroelastic Biot system coupled with a linearized Stokes flow. Subsequent
contribution \cite{BCMW:21} addressed interaction problems involving either dynamic linear Biot models or quasi-static nonlinear poroelastic Biot models
coupled with linear Stokes equations, again within a weak solution setting. A fully nonlinear coupling with the three-dimensional
Navier-Stokes equations was later considered in \cite{Ces:17}, the author proved the existence of weak solutions. Recently, the existence of a
weak solution to Navier-Stokes-poro(visco)elastic media interaction has been analyzed in \cite{KCM:24, BCM:25, KCM:25}. For further results, we refer to \cite{AGW:24}. 

Despite this
substantial body of work, uniqueness of weak solutions for fluid-Biot interaction models in the poro(visco)elastic situation remains open.

From a computational viewpoint, the interaction of fluids with poroelastic and viscoelastic materials has been extensively studied in the numerical literature \cite{BQQ:09,MWW:15,BYZ:15,AKYZ:18}.

It seems that there are no results addressing uniqueness and strong solutions in the nonlinear situation subject to  viscoelastic effects and
Beavers-Joseph-Saffman interface conditions.

The present article fills this gap. Our main result \autoref{thm:main} proves local-in-time existence and uniqueness of strong solutions
for arbitrary large initial data, as well as global-in-time existence and uniqueness for sufficiently small data. The solutions constructed satisfy the
interface conditions, including the Beavers-Joseph-Saffman condition, in a strong trace sense. We obtain the strong well-posedness even in
critical Besov spaces that are invariant under the natural scaling of the Navier--Stokes equations.
Within this framework, we also derive Serrin-type blow-up criteria in \autoref{cor:bjs blow up}.

We briefly comment about our approach. The analysis is challenging due to the simultaneous presence of nonlinear fluid effects,
mixed hyperbolic-parabolic behavior in the solid, and nonstandard interface coupling. A key ingredient is the spectral behaviour of the coupled linearized
fluid-porous media operator. In fact, due to the hyperbolic nature of the Biot dynamics, the resolvent of this operator cannot be  compact, and the
spectrum may contain an essential part. This precludes the use of standard compactness-based spectral arguments. 

A key observation of this work is that the presence of viscoelastic damping allows us to prove that the spectral bound of the
linearized operator is strictly negative, despite the lack of compactness and the possible presence of essential spectrum.
Combining this fact with regularity properties of the linearized system allows then to obtain our strong well-posedness result.

A main difficulty here is to locate
the essential spectrum. This is mainly done by the observation that the spectrum of the linearized operator is
closely related to the spectrum of an associated operator $C=M((-L)^{1/2})$, where $M(\alpha)$ is the $2\times 2$-matrix given by 
$			M(\alpha) := 
			\begin{pmatrix}
				-\delta^{-1} & \alpha \\
				- \delta^{-2} \alpha^{-1} & -\delta \alpha^2 + \delta^{-1}
			\end{pmatrix} $,  $L$ denotes the Lam\'e operator and $\delta>0$ is the visco-elastic parameter as above.
Then, surprisingly, by a functional calculus and spectral mapping type approach, spectral properties of $C$ can be deduced elegantly from the ones of $M$, which are of course accessible more easily.


The article is organized as follows. In \autoref{sec2} we introduce the coupled fluid-poro-viscoelastic model. Basic notation and preliminary results are collected in \autoref{sec:preliminaries}. The main well-posedness results, including the blow-up criteria, are stated in \autoref{sec:main}.
In \autoref{sec:visco} the coupled Navier-Stokes-viscoelastic system is reformulated as a semilinear evolution equation. The linearization is analyzed in \autoref{sec:representation}, where it is shown via a suitable similarity transformation and boundary perturbation arguments that it can be reduced to a simpler operator matrix. This representation is then used to characterize its interpolation spaces.
The spectral properties of the linearized operator are studied in \autoref{sec:spectral theory}, where we prove that its spectral bound is strictly negative. Building on these results, unshifted maximal regularity is established in \autoref{sec:mr}, which forms the key ingredient for global strong well-posedness for small initial data.
Finally, in \autoref{sec:wellposedness}, the maximal regularity framework is combined with classical bilinear estimates for the Navier-Stokes nonlinearity and a fixed-point argument to prove the main theorem and the stated blow-up criteria.
	
		\section{Description of the model}\label{sec2}

The Biot model provides a macroscopic description of a porous elastic skeleton
saturated by a viscous fluid.
In its dynamic formulation, the model couples the balance of linear momentum for
the solid matrix with a diffusion equation for the pore pressure.
Denoting by $\up$ the displacement of the porous elastic skeleton and by $\pp$ the pore pressure,
the Biot system reads 
\begin{equation} \label{eq:biot} \tag{Biot}
		\left\{
		\begin{aligned}
			\partial_{tt} \up-\D\sigmapdelta(\up,\pp) &=0
			&&\text{ in } (0,T)\times\Omegap,\\
			\partial_t (\pp+\alpha\,\div\up) -k\Delta \pp&=0 &&\text{ in } (0,T)\times\Omegap ,
		\end{aligned}
		\right.
	\end{equation}
Here $\Omega_p\subset\mathbb{R}^3$ denotes the region occupied by the porous medium,
$\alpha>0$ is the Biot--Willis coefficient measuring the coupling between volumetric
strain and pore pressure, and $k>0$ is the permeability parameter. The total stress tensor in the porous solid is given by
\begin{equation*}
		\sigmapdelta(\up,\pp)=\lambdap \varepsilon\bigg(\up+\delta \cdot \partial_t \up\bigg) + 2\mup\varepsilon\bigg(\up+\delta \cdot\partial_t \up\bigg)-\alpha \pp\mathbb{I},
	\end{equation*}
where $\varepsilon(u)=\tfrac12(\nabla u+\nabla u^{\top})$ denotes the linearized strain
tensor and $\lambdap,\mup>0$ are the Lam\'e coefficients of elasticity. The parameter $\delta\ge 0$ determines the qualitative nature of the solid response.
When $\delta=0$, we arrive at the classical dynamic \emph{poroelastic} Biot system.
In contrast, $\delta>0$ captures \emph{viscoelastic} effects and can be interpreted as a
Kelvin-Voigt type regularization of the elastic operator.

Although the displacement equation exhibits parabolic-like dissipation due to
Kelvin-Voigt damping, it remains a second-order hyperbolic equation, whereas the
pressure equation is purely parabolic. In particular, the natural regularity for $\up$ is limited to $\rH^{1,q}(\Omegap)$, even in the presence of damping. This limited regularity plays a decisive role in the analysis of the coupled problem, especially in the presence of interface conditions.

The porous medium is surrounded by a fluid region. The dynamics of the fluid are coupled to the porous medium through transmission conditions at the interface. The fluid occupies a bounded domain $\Omegaf\subset\mathbb{R}^3$ and is governed by
the incompressible Navier-Stokes equations for the fluid  velocity $\uf$ and pressure $\pf$:
\begin{equation} \label{eq:nse} \tag{NSE}
		\left\{
		\begin{aligned}
			\partial_t \uf-\D\sigmaf(\uf,\pf)+\uf \cdot \nabla \uf&=0 &&\text{ in } (0,T)\times\Omegaf,\\
			\operatorname{div}\uf&=0 &&\text{ in } (0,T)\times\Omegaf ,
		\end{aligned}
		\right.
	\end{equation} 
	Here $\sigmaf(\uf,\pf)= 2\varepsilon(\uf)- \pf\mathbb{I}$ is the Cauchy stress tensor. The most complicated part in the coupled poro-viscoelastic-fluid model is the interaction at the interface between the porous material and fluid region. The coupling between the fluid and the porous solid is imposed on the interface
$\Gamma=\overline{\Omegap}\cap\overline{\Omegaf}$ via transmission conditions
expressing continuity of the normal fluid flux, balance of normal stresses, and tangential slip of
Beavers-Joseph-Saffman type:
\begin{equation} \label{eq:bjs} \tag{ICC}
		\left\{
		\begin{aligned}
			\uf\cdot {n}_{\Gamma}&=\left(\partial_{t} \up-k\nabla \pp\right)\cdot {n}_{\Gamma},\\
			\sigmaf{n}_{\Gamma}&=\sigma_p^{\delta}{n}_{\Gamma},\\
			{n}_{\Gamma}\cdot\sigmaf{n}_{\Gamma}&=-\pp,\\
			\sum_{j=1}^2{t}^{j}_{\Gamma}\cdot\sigmaf{n}_{\Gamma}&=-\sum_{j=1}^2\beta(\uf-\partial_{t} \up)\cdot {t}^{j}_{\Gamma},
		\end{aligned}
		\right.
	\end{equation}
	where $\beta>0$ is the \emph{resistance parameter} in the tangential direction. Here $n_\Gamma$ denotes the outward pointing normal at $\Gamma$ with respect to $\Omegap$ and $t_\Gamma^j$ two linear independent tangential vectors of $\Gamma$. The boundary conditions in \eqref{eq:bjs}$_4$ are the so called \emph{Beavers-Joseph-Saffman boundary conditions}, see \cite{BJ:67,Saf:71}. Roughly speaking it says that the tangential component of the stress of the fluid flow is proportional to the jump in the tangential velocity over the interface. For a rigorous derivation of the Beavers-Joseph-Saffman  condition based on stochastic homogenization, we refer to the work of J\"ager and Mikelic \cite{JM:00}.   From the functional--analytic point of view, these interface conditions lead to a
linearized operator with non--diagonal domain and prevent a decoupling of the fluid
and solid subproblems.
Moreover, they impose compatibility constraints which significantly complicates the spectral and
regularity analysis.
	
	The system \eqref{eq:biot}-\eqref{eq:nse}-\eqref{eq:bjs} is supplemented by additional boundary conditions on the non-interface part of the boundary as well as initial data. In order to make this precise, we first describe the geometry of the situation under consideration.
	The domains occupied by the viscoelastic material $\Omegap$ and by the fluid $\Omegaf$ are non-empty, bounded domains in $\R^3$ with smooth boundaries $\partial \Omegap$ and $\partial \Omegaf$, respectively. Moreover, we assume that there exists a non-empty bounded domain $\Omega_i \subset \Omegap$ with smooth boundary $\Gammap := \partial \Omega_i$.  
	Furthermore, we assume that the interface $\Gamma$ between the poro-viscoelastic material and the fluid does not touch the inner boundary $\Gammap$ and the outer boundary $\Gammaf := \partial \Omegaf \setminus \Gamma$.
	%
	
	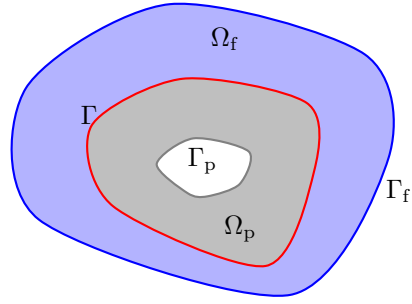
\begin{figure}[H]\label{fig}
		\begin{tikzpicture}[>=Triangle,
			dot/.style={circle,fill=blue!70!black,inner sep=1.3pt}]
			
			\draw[blue,thick,fill=blue!30,xshift=1cm] plot[smooth cycle]
			coordinates {(10:2.5) (40:2.8) (100:2.5) (150:2.7)   (190:2.2)
				(310:1.8)};
				
			\draw[red,thick,fill=gray!50,xshift=1cm] plot[smooth cycle]
			coordinates {(10:1.5) (40:1.8) (100:1.5) (150:1.7)   (190:1.2)
				(310:1.3)};
				
			\draw[gray,thick,fill=white!100,xshift=1cm] plot[smooth cycle]
			coordinates {(10:0.5) (40:0.8) (100:0.7) (150:0.7)   (190:0.2)
				(310:0.1)};
			
			\node [align=left] at (1.5,-0.5) {$\Omegap$};
			\node [align=left] at (1,0.4) {$\Gammap$};
			\node [align=left] at (3.6,0) {$\Gammaf$};
			\node [align=left] at (-0.5,1) {$\Gamma$};
			\node [align=left] at (1.3,2) {$\Omegaf$};	
			\end{tikzpicture}
		\caption{
		Prototype Geometry}
	\end{figure}
	The boundary conditions at the inner and outer boundary are homogeneous Dirichlet boundary conditions 	
	\begin{equation} \label{eq:bdry}
		\uf=0\quad \mbox{on}\quad (0,T)\times\Gammaf, \quad \up=0 \tand  \pp=0 \quad \mbox{on}\quad (0,T)\times\Gammap.
	\end{equation}
	We emphasis that the Biot model is a coupled \emph{hyperbolic-parabolic system}. It is (damped) hyperbolic and second order in time with respect to the velocity $\up$ and parabolic with respect to the pressure $\pp$. Hence, the system \eqref{eq:biot}-\eqref{eq:nse}-\eqref{eq:bjs} is a coupled nonlinear hyperbolic-parabolic-parabolic system. 
	Therefore, we have to prescribe the following initial data 
	\begin{equation}\label{eq:initial}
		\up(0)={u}_0, \ \partial_{t} \up(0)={v}_0, \  \pp(0)=p_0, \tand \uf(0)={u}^{f}_0.
	\end{equation}
	
	\section{Functional analytic framework}
	\label{sec:preliminaries}
	
		We start by introducing time weighted Banach valued spaces $\rL^p_\mu(J;E)$. More precisely, for a Banach space $(E,\|\cdot\|_E)$, $J = (0,T)$, $0 < T \le \infty$, $p \in (1,\infty)$, $\mu \in (\frac{1}{p},1]$ as well 
	as $u \colon J \to E$, we denote by $t^{1-\mu} u$ the function $t \mapsto t^{1-\mu} u(t)$ on $J$. We then define
	\begin{equation*}
		\rL_\mu^p(J;E) := \{u \colon J \to E : t^{1-\mu} u \in \rL^p(J;E)\}.
	\end{equation*}
	The latter space becomes a Banach space when equipped with the norm
	\begin{equation*}
		\|u \|_{\rL_\mu^p(J;E)} := \big\| t^{1-\mu} u \big \|_{\rL^p(J;E)} = \Big(\int_J t^{p(1-\mu)} \| u(t) \|_E^p dt\Big)^{\frac{1}{p}}.
	\end{equation*}
	For $k \in \mathbb{N}_0$, the associated weighted Sobolev spaces are defined by
	\begin{equation*}
		\rW_\mu^{k,p}(J;E) = \rH_\mu^{k,p}(J;E) := \left\{u \in \rW_\mathrm{loc}^{k,1}(J;E) : u^{(j)} \in \rL_\mu^p(J;E), \enspace j \in \{0,\dots,k\}\right\},
	\end{equation*}
	and these spaces are equipped with the norms
	\begin{equation*}
		\| u \|_{\rW_\mu^{k,p}(J;E)} = \| u \|_{\rH_\mu^{k,p}(J;E)} := \Big(\sum_{j=0}^k \| u^{(j)} \|_{\rL_\mu^p(J;E)}^p \Big)^{\frac{1}{p}}.
	\end{equation*}
	The weighted fractional Sobolev spaces and Bessel potential spaces are then defined by interpolation.
	
	Let $X_0, X_1$ be two Banach spaces such that $X_1 \hookrightarrow X_0$ is densely embedded. We denote by $X_\beta := [X_0,X_1]_{\beta}$ the complex interpolation spaces for $\beta \in (0,1)$ and by 
	$X_{\theta,r} := (X_0,X_1)_{\theta,r}$ the real interpolation spaces for $\theta \in (0,1)$. Furthermore, the trace space is given by $		X_{\gamma,\mu} := X_{\mu-\frac{1}{p},p} .$
	Moreover, the maximal $\rL^p$-regularity space is defined as
	\begin{equation*}
		\E_{1,\mu}(J) := \rH^{1,p}_\mu(J;X_0) \cap \rL^p_\mu(J;X_1)  
	\end{equation*}
	and the data space as
	\begin{equation*}
		\E_{0,\mu}(J) := \rL^p(J;X_0).
	\end{equation*}
	We obtain the following embedding
	\begin{equation}
		\E_{1,\mu}(J) \hookrightarrow \mathrm{BUC}(J;X_{\gamma,\mu}) . \label{eq:embedding mr}
	\end{equation}
	
	\medskip 
	
		Let $\Omega \subset \R^3$ be a domain with boundary $\partial \Omega$ and $\Gamma \subset \partial \Omega$. 
	For $s > 0$ and $s \not \in \mathbb{N}$ we denote by $\rH^{s,q}(\Omega)$ the Bessel potential spaces, by $\rW^{s,q}(\Omega)$ the fractional Sobolev spaces and by $\rB_{q,r}^s(\Omega)$ the Besov spaces 
	for $r \in [1,\infty]$. For the definition of these spaces we refer e.g. to \cite{Tri:78} or \cite{Ama:19}. The analogous spaces with respect to homogeneous Dirichlet boundary conditions at $\Gamma$ are 
	defined as  
	
	\begin{equation*}
		\rH_{\Gamma}^{s,q}(\Omega) 
		:=
		\begin{cases}
			\{ v \in \rH^{s,q}(\Omega) \colon v|_{\Gamma} = 0 \}, &\text{ if } s > \frac{1}{q}, \\ 
			\{ v \in \rH^{1/q,q}(\R^3) \colon \mathrm{supp}(v) \subset \Omega \cup \Gamma \}, &\text{ if } s = \frac{1}{q}, \\
			\qquad \rH^{s,q}(\Omega) &\text{ if } s < \frac{1}{q}.
		\end{cases} 
	\end{equation*}
	and
	\begin{equation*}
		\rB_{q,r,\Gamma}^{s}(\Omega) 
		:=
		\begin{cases}
			\{ v \in \rB_{q,r}^{s}(\Omega) \colon v|_{\Gamma} = 0 \}, &\text{ if } s > \frac{1}{q}, \\
			\{ v \in \rB_{q,r}^{1/q}(\R^3) \colon \mathrm{supp}(v) \subset \Omega \cup \Gamma \}, &\text{ if } s = \frac{1}{q}, \\
			\qquad \rB_{q,r}^{s}(\Omega) &\text{ if } s < \frac{1}{q}. \\
		\end{cases} 
	\end{equation*}
	Note, that
	\begin{equation*}
		\rB^{\frac{1}{q}}_{q,q,\Gamma}(\Omega) = \rW^{\frac{1}{q},q}_{00,\Gamma}(\Omega),
	\end{equation*}
	denotes the $\rL^q$-type {\em Lions-Magenes space} on $\Omega$ with respect to $\Gamma \subset \partial \Omega$ given by all $v \in \rW^{\frac{1}{q},q}(\Omega)$ such that 
	\begin{equation*}
		\int_{\Omega} \rho(x)^{-1} | v(x) |^q \, \mathrm{d}x < \infty. 
	\end{equation*}
	Here $\rho$ denotes a smooth function comparable to the distance $\mathrm{dist}(\cdot,\Gamma)$ from $\Gamma$, equipped with the norm
	\begin{equation*}
		\| v \|_{\rW_{00,\Gamma}^{\frac{1}{q},q}}
		= \biggl(\| v \|_{\rW^{\frac{1}{q},q}}^q  
		+ \int_{\Omega} \frac{|v(x)|^q}{\mathrm{dist}(x,\Gamma)^2} \, \mathrm{d}x \biggr)^{\frac{1}{q}} .
	\end{equation*}
	For more information on these spaces we refer e.g. to \cite{See:72}.
	In particular, for $q = 2$ we obtain the usual Lions-Magenes space  $\rH_{00,\Gamma}^{\frac{1}{2}}(\Omega)$ on $\Omega$ with respect to $\Gamma$, see \cite{LM:72}. 
	
	Moreover, motivated by the scaling invariant critical spaces of the Navier-Stokes equations we consider the following critical Besov spaces with boundary conditions
	\begin{equation*}
		\rB_{q,r,\Gamma}^{\frac{3}{q}-1}(\Omega) 
		:=
		\begin{cases}
			\{ v \in \rB_{q,r}^{\frac{3}{q}-1}(\Omega) \colon v|_{\Gamma} = 0 \}, &\text{ if } q > 2, \\
			\qquad \rH^{\frac{1}{2}}_{00,\Gamma}(\Omega), &\text{ if } q = 2, \\
			\qquad \rB_{q,r}^{\frac{3}{q}-1}(\Omega) &\text{ if } q < 2 . \\
		\end{cases} 
	\end{equation*}
	Finally, we define the solenoidal $\rL^q$-vector fields by
	\begin{equation*}
		\rL^q_{\sigma}(\Omega) = \overline{\{ u \in \rC^\infty(\Omega;\R^3) \colon \div(u) = 0 \}}^{\| \cdot \|_{\rL^q}}
	\end{equation*}
	for $q \in (1,+\infty)$. We obtain that $\rL^q_{\sigma}(\Omega) \subset \rL^q(\Omega;\R^3)$ is a closed subspace. Further, we denote the Helmholtz projection by
	\begin{equation*}
		\mathcal{P} \colon \rL^q(\Omega;\R^3) \to \rL^q_{\sigma}(\Omega) .
	\end{equation*}
	
	\section{Statement of the Main results}
	\label{sec:main}
	
	Our main theorem is the strong well-posedness of the coupled Navier-Stokes-viscoelastic system \eqref{eq:biot}-\eqref{eq:nse}-\eqref{eq:bjs}. 
	
	\begin{thm}[Strong well-posedness]\label{thm:main}
	Suppose $r, q \in (1,\infty)$ such that $\frac{2}{3 r} + \frac{1}{q} \leq 1$ and let $\mu_0 := \frac{1}{2}+\frac{1}{2q}+\frac{1}{r}$. 
	\begin{enumerate}[(a)]
		\item 
		Let $q < 3$, $\mu_c := \frac{3}{2q}-\frac{1}{2}+\frac{1}{r}$, $\mu \in [\mu_c,1] \setminus \{\mu_0\}$ and $({u}_0, {v}_0, p_0, {u}_0^f) \in \rH^{1,q}_{\Gammap}(\Omegap) \times \rB_{q,r,\Gammap}^{2\mu-\frac{2}{r}}(\Omegap) \times
		\rB_{q,r,\Gammap}^{2\mu-\frac{2}{r}}(\Omegap)\times
		\rB_{q,r,\Gammaf,\sigma}^{2\mu-\frac{2}{r}}(\Omegaf)$ for $\mu > \mu_0$.
		Then there exists $T>0$ such that the system \eqref{eq:biot}-\eqref{eq:nse}-\eqref{eq:bjs}, supplemented by \eqref{eq:bdry} and \eqref{eq:initial}, admits a unique, strong solution 
		\begin{align*}
			\up &\in 
			\rH^{1,r}_{\mu}(0,T; \rH^{1,q}(\Omegap)), \\
			\up+\delta \partial_t \up &\in \rL^r_{\mu}(0,T; \rH^{2,q}(\Omegap)) \cap \rH^{1,r}_{\mu}(0,T; \rL^q(\Omegap))
			\cap \mathrm{BUC}([0,T), \rB^{{2\mu-\frac{2}{r}}}_{q,r,\Gammap}(\Omegap)),\\
			\pp &\in \rL^r_{\mu}(0,T; \rH^{2,q}(\Omegap)) \cap \rH^{1,r}_{\mu}(0,T; \rL^q(\Omegap))
			\cap \mathrm{BUC}([0,T), \rB^{{2\mu-\frac{2}{r}}}_{q,r,\Gammap}(\Omegap)),\\
			\uf &\in 
			\rL^r_{\mu}(0,T; \rH^{2,q}(\Omegaf)) \cap \rH^{1,r}_{\mu}(0,T; \rL^q_{\sigma}(\Omegaf))
			\cap \mathrm{BUC}([0,T), \rB^{{2\mu-\frac{2}{r}}}_{q,r,\Gammaf,\sigma}(\Omegaf)), 
		\end{align*}
		Further, there exists $r_0 > 0$ such that for initial data $({u}_0, {v}_0, p_0, {u}_0^f)$ with 
		\begin{equation*} 
			\|{u}_0\|_{\rH^{1,q}_{\Gammap}(\Omegap)} + \|{v}_0\|_{\rB_{q,r,\Gammap}^{2\mu-\frac{2}{r}}(\Omegap)} + \| p_0 \|_{\rB_{q,r}^{2\mu-\frac{2}{r}}(\Omegap)} + \| u^{f}_0 \|_{\rB_{q,r}^{2\mu-\frac{2}{r}}(\Omegaf)} < r_0,
		\end{equation*}
		the unique, strong solution exists globally. 
		\item  
		Let $q \geq 3$, $\mu \in (\frac{1}{r},1]\setminus \{\mu_0\}$ and $({u}_0, {v}_0, p_0, {u}_0^f) \in \rH^{1,q}_{\Gammap}(\Omegap) \times \rB_{q,r,\Gammap}^{2\mu-\frac{2}{r}}(\Omegap) \times
		\rB_{q,r,\Gammap}^{2\mu-\frac{2}{r}}(\Omegap)\times \rB_{q,r,\Gammaf,\sigma}^{2\mu-\frac{2}{r}}(\Omegaf)$ satisfying \eqref{eq:bjs}  for $\mu > \mu_0$.
		Then there exists $T>0$ such that the system \eqref{eq:biot}-\eqref{eq:nse}-\eqref{eq:bjs}, supplemented by \eqref{eq:bdry} and \eqref{eq:initial}, admits a unique, strong solution 
		\begin{align*}
			\up &\in 
			\rH^{1,r}_{\mu}(0,T; \rH^{1,q}(\Omegap)), \\
			\up+\delta \partial_t \up &\in \rL^r_{\mu}(0,T; \rH^{2,q}(\Omegap)) \cap \rH^{1,r}_{\mu}(0,T; \rL^q(\Omegap))
			\cap \mathrm{BUC}([0,T), \rB^{{2\mu-\frac{2}{r}}}_{q,r,\Gammap}(\Omegap)),\\
			\pp &\in \rL^r_{\mu}(0,T; \rH^{2,q}(\Omegap)) \cap \rH^{1,r}_{\mu}(0,T; \rL^q(\Omegap))
			\cap \mathrm{BUC}([0,T), \rB^{{2\mu-\frac{2}{r}}}_{q,r,\Gammap}(\Omegap)),\\
			\uf &\in 
			\rL^r_{\mu}(0,T; \rH^{2,q}(\Omegaf)) \cap \rH^{1,r}_{\mu}(0,T; \rL^q_{\sigma}(\Omegaf))
			\cap \mathrm{BUC}([0,T), \rB^{{2\mu-\frac{2}{r}}}_{q,r,\Gammaf,\sigma}(\Omegaf)), 
		\end{align*}
		Further, there exists $r_0 > 0$ such that for initial data $({u}_0, {v}_0, p_0, {u}_0^f)$ with 
		\begin{equation*} 
			\|{u}_0\|_{\rH^{1,q}_{\Gammap}(\Omegap)} + \|{v}_0\|_{\rB_{q,r,\Gammap}^{2\mu-\frac{2}{r}}(\Omegap)} + \| p_0 \|_{\rB_{q,r,\Gammap}^{2\mu-\frac{2}{r}}(\Omegap)} + \| u^{f}_0 \|_{\rB_{q,r,\Gammaf}^{2\mu-\frac{2}{r}}(\Omegaf)} < r_0,
		\end{equation*}
		the unique, strong solution exists globally.  
	\end{enumerate}
	\end{thm}
	
	In the following comment we address additional features of the solutions in \autoref{thm:main}.
	
	\begin{rem}\label{rem:mr_standard}
		\begin{enumerate}[(a)]
		\item Due to the regularity of the solutions stated in \autoref{thm:main}, the interface conditions \eqref{eq:bjs} are satisfied in the sense of trace. More precisely, the interface conditions~\eqref{eq:bjs} hold as functions in $\rL^r_\mu(0,T;\rW^{2-\frac{1}{q},q}(\Gamma))$. In particular, the Beavers-Joseph-Saffman boundary conditions holds in that sense.  
		\item The solutions in \autoref{thm:main} admit additional regularity instantaneously by parabolic smoothing. We have $\up+\delta \partial_t \up, \ \pp \in \mathrm{BUC}([\varepsilon,T];\rB^{2-\frac{2}{r}}_{q,r,\Gammap}(\Omegap))$ and $\uf \in \mathrm{BUC}([\varepsilon,T];\rB^{2-\frac{2}{r}}_{q,r,\Gammaf,\sigma}(\Omegaf))$ for all $\varepsilon > 0$. 
		Note that $\up$ itself does not regularizes because of the hyperbolic behaviour of the Biot system \eqref{eq:biot}. 
		\item 
		Using Angenent's parameter trick, see \cite{Ang:90a} and \cite[Section 9.4]{PS:16}, one can show that the solutions in \autoref{thm:main} are analytic in time and spaces: We have $\up+\delta \partial_t \up, \ \pp \in \rC^\omega((0,T) \times \Omegap))$ and $\uf \in \rC^\omega((0,T) \times \Omegaf)$. Again the variable $\up$ is merely in $\rH^{1,r}_\mu(0,T;\rH^{1,q}(\Omegap))$ due to the hyperbolic behaviour of the Biot system \eqref{eq:biot}. 
		\item If $\frac{4}{p}+\frac{3}{q} < 1$, then the global solution from \autoref{thm:main} for small initial data decays exponentially fast in $\rB^{2-\frac{2}{r}}_{q,r}$. The rate of exponential convergences is determined by the spectral bound of the linearization, which can be found in \autoref{sec:spectral theory}. 
		\end{enumerate} 
	\end{rem}

	Let us now briefly discuss the difference between the two cases in \autoref{thm:main}. 
	
	\begin{rem}\label{rem:critical spaces}
		The Navier-Stokes equations on $\R^3$ admit the scaling
		$\uf^\lambda = \lambda \uf(\lambda^2t,\lambda x)$. This yields the scaling invariant (critical) Besov spaces $\rB^{\frac{3}{q}-1}_{q,r}$. Choosing $\mu = \mu_c = \frac{3}{2q}-\frac{1}{2}+\frac{1}{r}$ critical in \autoref{thm:main}(a) yields $\rB^{2\mu-\frac{2}{r}}_{q,r} = \rB^{\frac{3}{q}-1}_{q,r}$, the scaling invariant critical Besov spaces. Note that the condition $q < 3$ in (a) guarantees that the order of the Besov spaces 
		$\rB^{\frac{3}{q}-1}_{q,r}$ is positive $\frac{3}{q}-1 > 0$. In contrast to this the initial data in part (b) of \autoref{thm:main} are merely subcritical $\rB^{2\mu-\frac{2}{r}}_{q,r} \subset \rB^{\frac{3}{q}-1}_{q,r}$. 
		The reason for this is that we restrict ourself here to strong solutions. If $q > 3$ the critical Besov spaces 
		$\rB^{\frac{3}{q}-1}_{q,r}$ are of negative order $\frac{3}{q}-1 < 0$ and therefore no longer contained in the $\rL^q$-spaces. 
		\medskip 
	\end{rem}

		We finish this section with critical blow-up criteria. 	
	\begin{cor}\label{cor:bjs blow up} 
		Suppose $q < 3$.
		Assume that $({u}_0, {v}_0, p_0, {u}_0^f) \in \rH^{1,q}_{\Gammap}(\Omegap) \times \rB_{q,r,\Gammap}^{\frac{3}{q}-1}(\Omegap) \times
		\rB_{q,r,\Gammap}^{\frac{3}{q}-1}(\Omegap)\times \big(\rB_{q,r,\Gammaf}^{\frac{3}{q}-1}(\Omegaf) \cap \rL^q_{\sigma}(\Omegaf)\big)$ and let $(\up,\pp,\uf)$ be the unique, local solution from \autoref{thm:main} on the maximal time interval $[0,t_+)$.
		\begin{enumerate}[(a)]
			\item 
			If $(\up,\pp,\uf)$ blows up in finite time, i.e. $t_+ < \infty$, then
			\begin{equation*}
				\lim_{t \uparrow t_+} 
				\biggl( \| \up(t) \|_{\rH^{1,q}} 
				+\| \up(t) + \delta \partial_t \up(t)  \|_{\rB^{\frac{3}{q}-1}_{q,r,\Gammap}}
				+\| \pp(t) \|_{\rB^{\frac{3}{q}-1}_{q,r,\Gammap}} \hspace{-0.5em} + \| \uf(t) \|_{\rB^{\frac{3}{q}-1}_{q,r,\Gammaf}} \biggr) 
				\quad \text{does not exist.}
			\end{equation*}
			\item
			If $(\up,\pp,\uf)$ blows up in finite time, i.e. $t_+ < \infty$, then 
			\begin{equation*}
				\limsup_{t \uparrow t_+} \int_0^{t} 
				\biggl(\| \up(s) \|_{\rH^{1,q}}^r
				+\| \up(t) + \delta \partial_t \up(t) \|_{\rH^{2\mu_c,q}}^r
				+\| p(s) \|_{\rH^{2\mu_c,q}}^r + \| u(s) \|_{\rH^{2\mu_c,q}}^r \mathrm{d} s \biggr) = + \infty ,
			\end{equation*}
			where $\mu_c := \frac{3}{2q}+\frac{1}{r}-\frac{1}{2}$.
		\end{enumerate} 
	\end{cor}	
	
	\section{The Navier-Stokes-Viscoelastic Model as a semilinear Evolution equation}
	\label{sec:visco}
		
	In this section we reformulate the Navier-Stokes-viscoelastic system \eqref{eq:biot}--\eqref{eq:nse}--\eqref{eq:bjs} with \eqref{eq:bdry} and \eqref{eq:initial} as an evolution equation. To this end, we first introduce a new variable ${v}_p=\partial_t \up$ and rewrite it as a system of first order equations
	\begin{equation} \label{eq:ns-visco2}
		\left\{
		\begin{aligned}
			\partial_{t} \up &= {v}_p 
			&&\text{ in } (0,T)\times\Omegap,\\
			\partial_{t} {v}_p-\D\sigma_p^{\delta}&=0
			&&\text{ in } (0,T)\times\Omegap,\\
			\partial_t \pp + \alpha \operatorname{div}
			{v}_p
			-k\Delta \pp&=0 &&\text{ in } (0,T)\times\Omegap , \\
			\partial_t \uf-\D\sigmaf+\uf \cdot \nabla \uf 
			&=0 &&\text{ in } (0,T)\times\Omegaf,\\
			\operatorname{div}\uf&=0 &&\text{ in } (0,T)\times\Omegaf, 
		\end{aligned}
		\right.
	\end{equation}
	along with the following interface conditions
	\begin{equation} \label{eq:ns-visco3}
		\left\{
		\begin{aligned}
			\uf\cdot {n}_{\Gamma}&=({v}_p-k\nabla \pp)\cdot {n}_{\Gamma} &&\text{ on }(0,T)\times \Gamma,\\
			\sigmaf{n}_{\Gamma}&=\sigma_p^{\delta}{n}_{\Gamma} &&\text{ on }(0,T)\times \Gamma,\\
			{n}_{\Gamma}\cdot\sigmaf{n}_{\Gamma}&=-\pp &&\text{ on }(0,T)\times \Gamma,\\
			{t}^{j}_{\Gamma}\cdot\sigmaf{n}_{\Gamma}&=-\beta(\uf-{v}_p)\cdot {t}^{j}_{\Gamma},&&\text{ on }(0,T)\times \Gamma.
		\end{aligned}
		\right.
	\end{equation}
	From \eqref{eq:bdry} we have the following boundary conditions
	\begin{equation} \label{eq:ns-visco4}
		\uf = 0 \text{ on }(0,T)\times \Gammaf ,\quad \up=0 \quad \mbox{on}\quad (0,T)\times\Gammap,\quad \pp = 0 \text{ on }(0,T)\times \Gammap ,
	\end{equation}
	and from \eqref{eq:initial} the initial conditions
	\begin{equation} \label{eq:ns-visco5}
		\up(0)={u}_0, \ {v}_p(0)={v}_0, \  \pp(0)=p_0, \ \uf(0)={u}^{f}_0.
	\end{equation}
	For $q \in (1,\infty)$, we consider
	\begin{equation}
		\Xn= \rH^{1,q}(\Omegap)\times \rL^q(\Omegap) \times \rL^q(\Omegap)\times \rL^q_{\sigma}(\Omegaf).
		\label{eq:X0}
	\end{equation}
	Now, we define the maximal Laplacian $\Delta_m:\mathrm{D}(\Delta_m) \subset \rL^q(\Omegap) \rightarrow \rL^q(\Omegap)$ as
	\begin{equation}\label{eq:Delta_m}
		\Delta_m p=\Delta p \quad \mbox{ with domain } \quad \mathrm{D}(\Delta_m)= \rH^{2,q}(\Omegap),
	\end{equation}
	and the maximal Stokes operator $A_m \colon \mathrm{D}(\Am) \subset \rL^q_{\sigma}(\Omegaf) \to \rL^q_{\sigma}(\Omegaf)$ as 
	\begin{equation}\label{viscoeq:stokes}
		\Am u= \mathcal{P} \Delta u
		\quad \text{ with domain } \quad 
		\mathrm{D}(\Am)=
		\rH^{2,q}(\Omegaf) \cap \rL^q_{\sigma}(\Omegaf) ,
	\end{equation}
	where $\mathcal{P}\colon \rL^q(\Omegaf) \to \rL^q_{\sigma}(\Omegaf)$ denotes the Helmholtz projection. Further, we define the maximal elasticity operator $\mc{L}_m \colon \mathrm{D}(\mc{L}_m) \subset \rL^q(\Omegap) \rightarrow \rL^q(\Omegap)$ as
	\begin{equation}\label{viscoeq:lame}
		\mc{L}_m({u})
		=\D (2\mup\varepsilon({u})+\lambdap \D{u}\mathbb{I})
		\quad 
		\text{ with domain }
		\quad \mc{D}(\mc{L}_m) = \rH^{2,q}(\Omegap).
	\end{equation}
	Let us emphasize that the domains of these operators does not contain any boundary conditions. We incorporate them in the next step.
	Note that the conditions \eqref{eq:ns-visco3}$_3$ and \eqref{eq:ns-visco3}$_4$ describe the normal part and the tangential part of $\sigmaf n_\Gamma$, respectively. Thus, it follows
	\begin{equation}\label{visco:Cpf4}
		\sigmaf{n}_{\Gamma}=({n}_{\Gamma}\cdot\sigmaf{n}_{\Gamma}){n}_{\Gamma}+\sum_{l=1}^{2}(t^{l}_{\Gamma}\cdot \sigmaf{n}_{\Gamma})t^{l}_{\Gamma} = -\pp{n}_{\Gamma} - \beta\sum_{l=1}^{2}((\uf-{v}_p)\cdot t^{l}_{\Gamma})t^{l}_{\Gamma}.
	\end{equation}
	Now, we define the operator matrix $\bA \colon \mathrm{D}(\bA) \subset \Xn \to \Xn$ by
	\begin{equation}\label{def:Cpfdelta}
		\bA=
		\begin{pmatrix}
			0 & 1 & 0 &0\\ \mc{L}_m & \delta\mc{L}_m & -\alpha \nabla &0 \\ 0 & - \alpha \operatorname{div}& k\Delta_m & 0 \\ 0 & 0& 0 & A_m
		\end{pmatrix},  
	\end{equation}
	with non-diagonal domain
	\begin{multline}\label{eq:X1}
		\Xe=\Big\{(\up, \vp, \pp,\uf)\in \rH^{1,q}(\Omegap)\times \rH^{1,q}(\Omegap) \times {D}(\Delta_m)\times {D}(A_m)\mid \mc{L}_m(\up+\delta \vp)\in \rL^q(\Omegap),\\
		\sigmapdelta n_{\Gamma} = \sigmaf n_{\Gamma} = -\pp{n}_{\Gamma} - \beta\sum_{l=1}^{2}((\uf-{v}_p)\cdot t^{l}_{\Gamma})t^{l}_{\Gamma},\quad 
		\uf\cdot {n}_{\Gamma}=({v}_p-k\nabla \pp)\cdot {n}_{\Gamma}, \quad u_{p}|_{\Gammap} = 0, \quad
		p_{p}|_{\Gammap} = 0, \quad u_{f}|_{\Gammaf} = 0 
		\Big\}.
	\end{multline}

	By setting $\bw := (\up,\vp,\pp,\uf)^\top$ and $\bF(\bw',\bw)
	:= (0,0,0, -\mathcal{P} ({u'}_f \cdot \nabla \uf))^\top$ 
	the system \eqref{eq:ns-visco2} supplemented to \eqref{eq:ns-visco3}--\eqref{eq:ns-visco4} and \eqref{eq:ns-visco5} can be reformulated as  
	\begin{equation}\label{viscoop:Cpf}
		\frac{\d}{\d t} \bw
		=\
		\bA \bw + \bF(\bw,\bw),
		\quad  		
		\bw(0)=( {u}_0, {v}_0, p_0, {u}^{f}_0)^\top.
	\end{equation}
	on the ground space $\Xn$.

	\section{Linearization}
	\label{sec:representation}	
	
	In this section we analyze the linearization given of $\bA$.
	Let us point out that there are two difficulties in the analysis of $\bA$. The first one is the fact that $u$ and $v$ are merely in $\rH^{1,q}(\Omegap)$ and not in $\mathrm{D}(\mc{L}_m)$, the second one is the coupling via the boundary conditions \eqref{eq:ns-visco3} and \eqref{eq:ns-visco4}. To deal with the first problem, we use the fact that $\up + \delta \vp \in \mathrm{D}(\mc{L}_m)$ and the idea of suitable similarity transformations. Define the isomorphism $\bS$ on $\Xn$ by
	\begin{equation*}
		\bS := 
		\begin{pmatrix}
			1 & 0 & 0 & 0 \\
			\delta^{-1} & 1 & 0 & 0 \\
			0 & 0 & 1 & 0 \\
			0 & 0 & 0 & 1 
		\end{pmatrix}
		\qquad \text{ with inverse }
		\qquad 
		\bS^{-1}
		:=
		\begin{pmatrix}
			1 & 0 & 0 & 0 \\
			-\delta^{-1} & 1 & 0 & 0 \\
			0 & 0 & 1 & 0 \\
			0 & 0 & 0 & 1 
		\end{pmatrix} .
	\end{equation*}
	We obtain that $\bA$ is isomorphic to $\bB := \bS \bA \bS^{-1}$ given by
	\begin{equation}
		\bB
		=\begin{pmatrix}
			-\delta^{-1} & 1 & 0 &0\\ 
			-\delta^{-2} &\delta \Lm + \delta^{-1} & - \alpha \nabla & 0 \\ \delta^{-1} \alpha \div & - \alpha \div& k\Deltam & 0 \\ 0 & 0& 0 & \Am
		\end{pmatrix},
		\label{eq:tA}
	\end{equation}
	with domain
	\begin{equation*}
		\mathrm{D}(\bB)=\left\{\mathbf{w}=(\up,\vp,\pp,\uf) \in \rH^{1,q}(\Omegap) \times \mathrm{D}(\Lm) \times \mathrm{D}(\Deltam) \times \mathrm{D}(\Am) \left| 
		 \begin{aligned} 
		 \mathbf{w}\text{ satisfies } \eqref{eq:bjs'} \qquad  \text{ and } \qquad  \\
		 p_{p}|_{\Gammap} = 0, u_{p}|_{\Gammap} = 0, u_{f}|_{\Gammaf} = 0
		 \end{aligned} 
		 \right.
		\right\},
	\end{equation*}
	where the transformed interface coupling conditions are given by
	\begin{equation} 
	\left\{
	\begin{aligned}
		\varepsilon(\vp)n_\Gamma &=(\lambdap+2\mup)^{-1}[ (\alpha-1) \pp n_\Gamma + \beta\delta^{-1} \sum_{l=1}^2 (\vp \cdot t_\Gamma^l)t_\Gamma^l 
		- \beta\delta^{-1} \sum_{l=1}^2 (\up \cdot t_\Gamma^l)t_\Gamma^l - \beta \sum_{l=1}^2 (\uf \cdot t_\Gamma^l)t_\Gamma^l] ,\\ 
		\sigmaf n_\Gamma &= - \pp n_\Gamma + \beta\delta^{-1} \sum_{l=1}^2 (\vp \cdot t_\Gamma^l)t_\Gamma^l 
		- \beta\delta^{-1} \sum_{l=1}^2 (\up \cdot t_\Gamma^l)t_\Gamma^l - \beta \sum_{l=1}^2 (\uf \cdot t_\Gamma^l)t_\Gamma^l, \\
		\partial_{n_\Gamma} \pp &= k^{-1}\delta^{-1}( \vp  -  \up) n_\Gamma - k^{-1} \uf n_\Gamma .
	\end{aligned}
	\right.
	\label{eq:bjs'}
	\end{equation} 
		
	We denote by $\Ln : \mathrm{D}(\Ln) \subset \rL^q(\Omegap) \rightarrow \rL^q(\Omegap)$, the elasticity operator with homogeneous boundary conditions,~i.e,
	\begin{equation*}
		\Ln({u})
		=\D (2\mu\varepsilon({u})+\lambda \D{u}\mathbb{I})
		\quad 
		\text{ with domain }
		\quad  \mathrm{D}(\Ln) := \{ \vp \in \rH^{2,q}(\Omegap) \colon \vp|_{\Gammap} = 0 \text{ and } \varepsilon(\vp)n_{\Gamma} = 0 \},
	\end{equation*}
	and by $\Delta_0 \colon \mathrm{D}(\Delta_0)  \subset \rL^q(\Omegap) \to \rL^q(\Omegap)$, the Laplacian with homogeneous boundary conditions, i.e.
	\begin{equation*}
		\Delta_0 \pp = \Delta \pp \quad  \mbox{ with domain } \quad \mathrm{D}(\Delta_0)= \{ \pp \in \rH^{2,q}(\Omegap) \mid \partial_{n_{\Gamma}} \pp = 0 \text{ and } \pp|_{\Gammap} = 0 \} ,
	\end{equation*}
	and by $A_0 \colon \mathrm{D}(A_0) \subset \rL^q_{\sigma}(\Omegaf) \to \rL^q_{\sigma}(\Omegaf)$ the Stokes operator with homogeneous boundary conditions, i.e.
	\begin{equation*}
		\An u = \mathcal{P} \Delta u
		\quad \text{ with domain } \quad 
		\mathrm{D}(\An)=
		\{ \uf \in \rH_{\sigma}^{2,q}(\Omegaf)\mid \sigmaf n_{\Gamma} = 0 \text{ and } \uf|_{\Gammaf} = 0 \}.
	\end{equation*} 
	Now, we consider the operator matrix $\bBn \colon \mathrm{D}(\bBn) \subset \Xn \to \Xn$ given by
	\begin{align}\label{def:A0}
		\bBn
		&=\begin{pmatrix}
			-\delta^{-1} & 1 & 0 &0\\ -\delta^{-2} & \delta \Ln+\delta^{-1} & -\alpha \nabla &0 \\ \delta^{-1}\alpha \div & - \alpha \div& k\Deltan & 0 \\ 0 & 0& 0 & \An
		\end{pmatrix}, 
	\end{align}
	with diagonal domain 
	\begin{equation}\label{def:domainA0}
		\mathrm{D}(\bBn)
		= \rH^{1,q}(\Omegap) \times \mathrm{D}(\Ln) \times \mathrm{D}(\Deltan) \times \mathrm{D}(\An).
	\end{equation}		
	We introduce the operator matrix $\Gn \colon \mathrm{D}(\Gn) \subset \rH^{1,q}(\Omegap) \times \rL^q(\Omegap) \to \rH^{1,q}(\Omegap) \times \rL^q(\Omegap)$, which models the Biot system part and it is given by
	\begin{equation*}
		\Gn := 
		\begin{pmatrix}
			-\delta^{-1} & 1 \\
			-\delta^{-2} & \delta \Ln+\delta^{-1} 
		\end{pmatrix},\quad \mbox{ with domain }
		\quad \mathrm{D}(\Gn)
		:=
		\rH^{1,q}(\Omegap) \times \mathrm{D}(\Ln).
	\end{equation*}
	Note that we can rewrite \eqref{def:A0}--\eqref{def:domainA0} as
	\begin{equation*}
		\bBn
		=\begin{pmatrix}
			\Gn & \binom{0}{-\alpha \nabla} & \binom{0}{0} \\ (\delta^{-1}\alpha \div, - \alpha \div)& k\Deltan & 0 \\ (0,0)& 0 & \An
		\end{pmatrix}, \quad \mbox{ with domain }
		\quad \mathrm{D}(\bBn) = \mathrm{D}(\Gn) \times \mathrm{D}(\Deltan) \times \mathrm{D}(\An) .
	\end{equation*} 
	
	In the next lemma we describe the real and complex interpolation spaces associated to $\bBn$.
	
	\begin{lem}\label{lem:interpolation spaces A0}
		We denote by $\Ze := \mathrm{D}(\bBn)$ equipped with the graph norm and by $\Zn := \Xn$. 
		\begin{enumerate}[(i)]
			\item For  $r \in (1,\infty)$, the real interpolation spaces $\bZ_{\theta,r} := (\Ze,\Zn)_{\theta,r}$
			are given by
			\begin{equation*}
				\bZ_{\theta,r}
				= 
				\left\{
				\begin{aligned}
					\begin{aligned}
						\rH^{1,q}(\Omegap) &\times \rB^{2\theta}_{q,r,\Gammap}(\Omegap) \times \rB^{2\theta}_{q,r,\Gammap}(\Omegap) \times \bigl( \rB^{2\theta}_{q,r,\Gammaf}(\Omegaf) \cap \rL^q_{\sigma}(\Omegaf) \bigr),
						&\text{ if } \theta \in (-1,\tfrac{1}{2}+\tfrac{1}{2q}),\\
						\rH^{1,q}(\Omegap) &\times \{\vp \in \rB^{2\theta}_{q,r,\Gammap}(\Omegap) \colon \varepsilon(v)n_\Gamma = 0 \} \times \{\pp \in \rB^{2\theta}_{q,r,\Gammap}(\Omegap) \colon \partial_{n_\Gamma} \pp = 0 \} \\ &\times \bigl\{ \uf \in \rB^{2\theta}_{q,r,\Gammaf}(\Omegaf) \cap \rL^q_{\sigma}(\Omegaf) \colon \sigmaf n_\Gamma = 0 \bigr\},
						&\text{ if } \theta \in (\tfrac{1}{2}+\tfrac{1}{2q},1).
					\end{aligned}
				\end{aligned}
				\right.
				.			
			\end{equation*}
			\item The complex interpolation spaces $\bZ_\theta := [\Ze,\Zn]_\theta$ are given by
			\begin{equation*}
				\bZ_\theta 
				:= 
				\left\{
				\begin{aligned}
					\begin{aligned}
						\rH^{1,q}(\Omegap) &\times \rH^{2\theta,q}_{\Gammap}(\Omegap) \times \rH^{2\theta,q}_{\Gammap}(\Omegap) \times \bigl( \rH^{2\theta,q}_{\Gammaf}(\Omegaf) \cap \rL^q_{\sigma}(\Omegaf) \bigr)
						&\text{ if } \theta \in (-1,\tfrac{1}{2}+\tfrac{1}{2q}),\\
						\rH^{1,q}(\Omegap) &\times \{\vp \in \rH^{2\theta,q}_{\Gammap}(\Omegap) \colon \varepsilon(v)n_\Gamma = 0 \} 
						\times \{\pp \in \rH^{2\theta,q}_{\Gammap}(\Omegap) \colon \partial_{n_\Gamma} \pp = 0 \} \\ &\times \bigl\{ \uf \in \rH^{2\theta,q}_{\Gammaf}(\Omegaf) \cap \rL^q_{\sigma}(\Omegaf) \colon \sigmaf n_\Gamma = 0 \bigr\}
						&\text{ if } \theta \in (\tfrac{1}{2}+\tfrac{1}{2q},1) .
					\end{aligned}
				\end{aligned}
				\right.
				.		
			\end{equation*}
		\end{enumerate}	
	\end{lem}
	\begin{proof}
		This follows from $\bZ_1 = \mathrm{D}(\bB_0) = \mathrm{D}(\bD_0)$ and $\Zn = \Xn$. 
	\end{proof}
	
	Using the notation $\Gm \colon \mathrm{D}(\Gm) \subset \rH^{1,q}(\Omegap) \times \rL^q(\Omegap) \to \rH^{1,q}(\Omegap) \times \rL^q(\Omegap)$ which models the Biot system part
	\begin{equation*}
		{\mathrm{G}}_m := 
		\begin{pmatrix}
			-\delta^{-1} & 1 \\
			-\delta^{-2} & \delta \Lm+\delta^{-1} 
		\end{pmatrix},\quad \mbox{ with domain }
		\quad \mathrm{D}({\mathrm{G}}_m)
		:=
		\rH^{1,q}(\Omegap) \times \mathrm{D}(\Lm),
	\end{equation*}
	we define the maximal operator $\bBm \colon \mathrm{D}(\bBm) \subset \bX \to \bX$ by
	\begin{equation}
		\bBm 
		= 
		\begin{pmatrix}
			\Gm & \binom{0}{-\alpha \nabla} & \binom{0}{0} \\
			(\delta^{-1}\alpha \div, -\alpha \div) & k \Deltam & 0 \\
			(0,0) & 0 & \Am 
		\end{pmatrix} , \quad 
		\mathrm{D}(\bBm) = \rH^{1,q}(\Omegap) \times \mathrm{D}(\Lm) \times \mathrm{D}(\Deltam) \times \mathrm{D}(\Am)
		\label{eq:tbAm}
	\end{equation}
	and note that $\bBn$ is a restriction of $\bBm$. 

	Next, we consider the elliptic boundary value problems associated to the Laplacian, the Lamé and the Stokes operators, respectively, given by
	\begin{equation}
		\left\{ 
		\begin{aligned}
			&\Deltam \pp = \lambda \pp, \\
			\partial_{n_{\Gamma}} \pp = &\varphi, \qquad 
			\pp|_{\Gammap} = 0 
		\end{aligned}
		\right.
		\label{eq:laplacian neumann problem}
	\end{equation}	
	for $\varphi \in \mathrm{W}^{1-\frac{1}{q},q}(\Gamma)$, and
	\begin{equation}
		\left\{ 
		\begin{aligned}
			&\Lm \vp = \lambda \vp, \\
			\varepsilon(\vp) n_{\Gamma} = &\eta, \qquad 
			\vp|_{\Gammap} = 0 
		\end{aligned}
		\right.
		\label{eq:Lame neumann problem}
	\end{equation}	
	for $\eta \in \mathrm{W}^{1-\frac{1}{q},q}(\Gamma)$,
	as well as,
	\begin{equation}
		\left\{ 
		\begin{aligned}
			&\Am \uf = \lambda \uf, \\
			\sigmaf n_{\Gamma} = &\psi, \qquad 
			\uf|_{\Gammaf} = 0 
		\end{aligned}
		\right.
		\label{eq:stokes neumann problem}
	\end{equation}	
	for $\psi \in \mathrm{W}^{1-\frac{1}{q},q}(\Gamma)$.	Finally, we consider the boundary problem associated to $\Gm$ given by
		\begin{equation}
		\left\{ 
		\begin{aligned}
			\Gm \tbinom{\up}{\vp} &= \lambda \tbinom{\up}{\vp}, \\
			\varepsilon(\vp) n_{\Gamma} = &\eta, \qquad 
			\vp|_{\Gammap} = 0 
		\end{aligned}
		\right.
		\label{eq:G neumann problem}
	\end{equation}	
	for $\eta \in \mathrm{W}^{1-\frac{1}{q},q}(\Gamma)$. We then obtain the following results.

	\begin{lem}\label{lem:L0}
		\begin{enumerate}[(i)]
			\item For all $\lambda \in \rho(\Deltan)$ and $\varphi \in \rW^{1-\frac{1}{q},q}(\Gamma)$ there exists a unique solution $\pp \in \rH^{2,q}(\Omegap)$ to \eqref{eq:laplacian neumann problem}. Moreover, it satisfies the estimate
			\begin{equation*}
				\| \pp \|_{\rH^{2,q}(\Omegap)} \leq C \cdot \| \varphi \|_{\rW^{1-\frac{1}{q},q}(\Gamma)} 
			\end{equation*}
			and the solution operator $\rL_\lambda^{\Deltam} \colon \rW^{1-\frac{1}{q},q}(\Gamma) \to \rH^{2,q}(\Omegap)$ of \eqref{eq:Lame neumann problem} is bounded. 
			\item For all $\lambda \in \rho(\Ln)$ and  $\eta \in \rW^{1-\frac{1}{q},q}(\Gamma)$ there exists a unique solution $\vp \in \rH^{2,q}(\Omegap)$ to \eqref{eq:Lame neumann problem}. Furthermore, it satisfies the estimate
			\begin{equation*}
				\| \vp \|_{\rH^{2,q}(\Omegap)} \leq C \cdot \| \eta \|_{\rW^{1-\frac{1}{q},q}(\Gamma)} 
			\end{equation*}
			and the solution operator $\rL_\lambda^{\Lm} \colon \rW^{1-\frac{1}{q},q}(\Gamma) \to \rH^{2,q}(\Omegap)$ of \eqref{eq:Lame neumann problem} is bounded. 
			\item For all $\lambda \in \rho(\An)$ and $\psi \in \rW^{1-\frac{1}{q},q}(\Gamma)$ there exists a unique solution $\uf \in \rH^{2,q}(\Omegaf) \cap \rL^q_{\sigma}(\Omegaf)$ to \eqref{eq:stokes neumann problem}. Moreover, it satisfies the estimate
			\begin{equation*}
				\| \uf \|_{\rH^{2,q}(\Omegaf)} \leq C \cdot \| \psi \|_{\rW^{1-\frac{1}{q},q}(\Gamma)} 
			\end{equation*}
			and the solution operator $\rL_\lambda^{\Am} \colon \rW^{1-\frac{1}{q},q}(\Gamma) \to \rH^{2,q}(\Omegaf) \cap \rL^q(\Omegaf)$ of \eqref{eq:stokes neumann problem} is bounded. 
			 \item For all $\lambda \in \rho(\Gn)$ and  $\eta \in \rW^{1-\frac{1}{q},q}(\Gamma)$ there exists a unique solution $\up, \vp \in \rH^{2,q}(\Omegap)$ to \eqref{eq:G neumann problem}. Furthermore, it satisfies the estimate
			\begin{equation*}
				\| \up \|_{\rH^{2,q}(\Omegap)}+\| \vp \|_{\rH^{2,q}(\Omegap)} \leq C \cdot \| \eta \|_{\rW^{1-\frac{1}{q},q}(\Gamma)} 
			\end{equation*}
			and the solution operator $\rL_\lambda^{\Gm} \colon \rW^{1-\frac{1}{q},q}(\Gamma) \to \rH^{2,q}(\Omegap)\times \rH^{2,q}(\Omegap)$ of \eqref{eq:G neumann problem} is bounded. 
		\end{enumerate}
	\end{lem}
	\begin{proof} As Lam\'{e} operator is parameter elliptic, we can use \cite{DHP:04} in order to obtain a bounded solution operator $\rL_\lambda^{\Lm} \colon \rW^{1-\frac{1}{q},q}(\Gamma) \to \rH^{2,q}(\Omegap)$ of \eqref{eq:Lame neumann problem}. By \cite[Corollary 7.4.5.]{PS:16},  we obtain bounded solution operators $\rL_\lambda^{\Deltam} \colon \rW^{1-\frac{1}{q},q}(\Gamma) \to \rH^{2,q}(\Omegap)$ of \eqref{eq:laplacian neumann problem} and $\rL_\lambda^{\Am} \colon \rW^{1-\frac{1}{q},q}(\Gamma) \to \rH^{2,q}(\Omegaf) \cap \rL^q(\Omegaf)$ of \eqref{eq:stokes neumann problem}.
		
	 In order to prove (iv) we 
	note that by \cite[Lemma 3.4]{AE:18} or \cite[Lemma 1.2]{Gre:87} it is sufficient to show that the operator $\rL^{\Gm}_{\lambda_0}$ exists and is bounded for some $\lambda_0 \in \rho(\Gn)$. 
	We may consider $\lambda_0 \in \rho(\Gn)$ such that $\mu_0 := \frac{\lambda_0^2}{\delta \lambda_0 + 1} \in \rho(\Ln)$ if we choose $\lambda_0$ sufficient large, since $\Ln$ and $\Gn$ are generators of $\rC_0$-semigroups. 
	We have $\binom{\up}{\vp} \in \ker(\lambda_0-\Gm)$ if and only if
	\begin{align*}
		-\delta^{-1} \up + \vp &= \lambda_0 \up, \\
		-\delta^{-2} \up 
		+ \delta \Lm \vp + \delta^{-1} \vp &= \lambda_0 \vp  .
	\end{align*} 
	The first equation guarantees that $\up$ and $\vp$ have the same regularity. Using the first equation, the second is equivalent to 
	\begin{equation*}
		\Lm \vp 
		= \delta^{-1} (\lambda_0 - \delta^{-1} + \delta^{-2}( \lambda_0+\delta^{-1} )^{-1}) \vp 
		=
		\frac{\lambda_0^2}{\delta \lambda_0 + 1} \vp 
		= \mu_0 \vp 
	\end{equation*}
	and $\ker(\lambda_0 -\Gm) \cong \ker(\mu_0 - \Lm)$.
	Now, the result follows from (ii).
	\end{proof}
	\begin{rem}
		Let us point out that the regularity of the solution $\up$ of \eqref{eq:G neumann problem} is one derivative better than one would expect firstly, due to the direct relation between $\up$ and $\vp$. This fact will play a crucial role in the spectral theory later. 
	\end{rem}
	In the next step we construct the solution operator $\rL_\lambda^{\bB}$ associated to $\bBm$ from \eqref{eq:tbAm} and the boundary operator $\text{L}(\up,\vp,\pp,\uf)^\top := (\varepsilon(\vp)n_\Gamma,\partial_{n_\Gamma}\pp, \sigmaf n_\Gamma)^\top$, i.e.
	\begin{equation}
		\bBm \mathbf{w}= \lambda \mathbf{w}, \qquad \mathrm{L} \mathbf{w}= (\varphi,\eta,\psi)^\top .
		\label{eq:neumann problem}
	\end{equation} 
	
	\begin{lem}\label{lem:L0_real}
		For all $\lambda \in \rho(\bBn)$ there exists a bounded solution operator 
		\begin{equation*} 
		\rL_\lambda^{\bB} \colon \rW^{1-\frac{1}{q},q}(\Gamma) \times \rW^{1-\frac{1}{q},q}(\Gamma) \times \rW^{1-\frac{1}{q},q}(\Gamma) \to \rH^{1,q}(\Omegap)\times \rH^{2,q}(\Omegap)\times \rH^{2,q}(\Omegap)\times \rH^{2,q}_{\sigma}(\Omegaf)
		\end{equation*} 
		of \eqref{eq:neumann problem}.
	\end{lem}
	\begin{proof}
		Using \cite[Lemma 3.4]{AE:18} or \cite[Lemma 1.2]{Gre:87} it is sufficient to show that $\rL_{\lambda_0}$ exists and is bounded for some $\lambda_0 \in \rho(\bBn)$. 
		We consider $\lambda_0 \in \rho(\bBn) \cap \rho(\Gn) \cap \rho(\Deltan) \cap \rho(\An)$. {The existence of such an element will be justified in \autoref{sec:spectral theory}.}
		Next, we split the operator $\bA_m$ from \eqref{eq:tbAm} into
		\begin{equation*}
			\bBm
			= 
			\bDm + \bQ
			\coloneqq 
			\begin{pmatrix}
				\Gm & 0 & 0 \\
				 0 & k \Deltam & 0 \\
				0 & 0 & \Am 
			\end{pmatrix}
			+
			\begin{pmatrix}
				0 & \binom{0}{-\alpha \nabla} & 0 \\
				(\delta^{-1}\alpha \div, -\alpha \div) & 0 & 0 \\
				(0,0) & 0 & 0 
			\end{pmatrix}
		\end{equation*}
		We define the combined solution operator  $\mathbf{L}^{\bD}_\lambda	
		:=  \diag(\rL^{\Gm}_\lambda,\rL^{\Deltam}_\lambda,\rL^{\Am}_\lambda)$ and note that it is the solution operator of 
		\begin{equation*}
			{\textbf{D}}_{\mathrm{m}} \mathbf{w}= \lambda \mathbf{w}, \qquad \mathrm{L} \mathbf{w}= (\varphi,\eta,\psi)^\top .
		\end{equation*}
		Since, 
		\begin{align*} 
		\mathrm{D}({\textbf{D}}_m) := \mathrm{D}(\bBm) &= \rH^{1,q}_{\Gammap}(\Omegap) \times \rH^{2,q}_{\Gammap}(\Omegap) \times \rH^{2,q}_{\Gammap}(\Omegap) \times (\rH^{2,q}_{\Gammaf}(\Omegaf) \cap \rL^q_{\sigma}(\Omegaf)) \\
		&\hookrightarrow
		\rH^{1,q}(\Omegap) \times 
		\rH^{1,q}(\Omegap) \times 
		\rH^{1,q}(\Omegap) \times 
		(\rH^{1,q}(\Omegaf)\cap \rL^q_{\sigma}(\Omegaf))
		\end{align*} 
		we see that $\bQ$ is relatively $\bDm$-bounded. Now \cite[Lemma 4.6]{BE:18} yields the existence and boundedness of $\rL^{\bB}_\lambda$ given by
		\begin{equation}
			\rL_\lambda^{\bB} 
			= \rL_\lambda^{\bD} + R(\lambda,\bB_0) \bQ \rL_\lambda^{\bD}
			\label{eq:formula Llambda}
			\qedhere 
		\end{equation}
	\end{proof}
	For $\bw= (\up, \vp, \pp,\uf)^{\top}$ we introduce the boundary operator
	\begin{equation*}
		\Phi(\bw)
		:= 
		\begin{pmatrix}
			(\lambdap+2\mup)^{-1}[ (\alpha-1) \pp n_\Gamma + \beta\delta^{-1} \sum_{l=1}^2 (\vp \cdot t_\Gamma^l)t_\Gamma^l 
			- \beta\delta^{-1} \sum_{l=1}^2 (\up \cdot t_\Gamma^l)t_\Gamma^l - \beta \sum_{l=1}^2 (\uf \cdot t_\Gamma^l)t_\Gamma^l]
			\\
			k^{-1}\delta^{-1}( \vp  -  \up) n_\Gamma - k^{-1} \uf n_\Gamma
			\\	
			- \pp n_\Gamma + \beta\delta^{-1} \sum_{l=1}^2 (\vp \cdot t_\Gamma^l)t_\Gamma^l 
			- \beta\delta^{-1} \sum_{l=1}^2 (\up \cdot t_\Gamma^l)t_\Gamma^l - \beta \sum_{l=1}^2 (\uf \cdot t_\Gamma^l)t_\Gamma^l 		
		\end{pmatrix}
			.
	\end{equation*}
	
	Using these operators we obtain the following representation.
	
	\begin{prop}\label{lem:representation}
		We have
		\begin{equation*}
			\bB
			= (\bB_{-\frac{1}{2}} + (\lambda-\bA_{-1}) \mathbf{L}^{\bB}_\lambda \Phi)|_{\Xn}, 
		\end{equation*}
		for each $\lambda \in \rho(\bBn)$,
		where $\bB_{-1}$ denotes the extrapolated operator associated to $\bBn$ and $\bB_{-\frac{1}{2}} = \bB_{-1}|_{\bZ_{-\frac{1}{2}}}$.
	\end{prop}
	\begin{proof}
		 We follow the lines of the proof of \cite[Lemma 3.4]{BHR:24}. 
		 Let us set $\bar{\bB}=(\bB_{-\frac{1}{2}} + (\lambda-\bB_{-1}) \mathbf{L}^{\bB}_\lambda \Phi)|_{\Xn}$ and fix $\lambda \in \rho(\bBn)$. As $\bB_{-\frac{1}{2}} \bw=\bB_{-1}\bw$ for $\bw\in \bZ_{\frac{1}{2}}$, we have
		\begin{align*}
			\bw\in \mathrm{D}(\bar{\bB}) \qquad 
			&\Longleftrightarrow \qquad 
			(\lambda-\bB_{-1}) (\mathbf{L}_\lambda^{\bB} \Phi-\Id) \bw\in \Xn \\
			&\Longleftrightarrow \qquad (\mathbf{L}_\lambda^{\bB} \Phi-\Id) \bw \in  \mathrm{D}(\bBn) = \ker(\mathbf{L}) \\
			&\Longleftrightarrow \qquad  \mathbf{L} \bw= {\Phi} \bw,
		\end{align*}
		where we used in the last line that $\bw = (\mathrm{Id}-\mathbf{L}^{\bB}_\lambda \Phi)\bw+\mathbf{L}^{\bB}_\lambda \Phi \bw\in \mathrm{D}(\bBn) + \ker(\lambda-\bBm) \subseteq \mathrm{D}(\bBm)$ such that
		\begin{equation*}
			0 = \mathbf{L}(\mathbf{L}^{\bB}_\lambda {\Phi} \bw- \bw)
			= {\Phi} \bw -  \mathbf{L} \bw.
		\end{equation*}
		Furthermore, for $\bw\in \mathrm{D}(\bBn)$ we obtain
		\begin{align*}
			\bar{\bB}
			\bw
			&= 
			(\bA_{-1} + (\lambda-\bB_{-1}) \mathbf{L}^{\bB}_\lambda \Phi) \bw\\
			&= 
			((\lambda-\bB) (\mathbf{L}^{\bB}_\lambda \Phi - \Id) \bw+ \lambda \bw
			\\
			&= 
			((\lambda-\bBm) (\mathbf{L}^{\bB}_\lambda \Phi - \Id) \bw+ \lambda \bw
			\\
			&= 
			-(\lambda-\bBm) \bw+ \lambda \bw
			= \bBm \bw = \bB \bw,
		\end{align*}
		where we used $(\lambda-\bBm)\mathbf{L}^{\bB}_\lambda \Phi \bw= 0$ and $\bB \subset \bBm$.
	\end{proof}
	
%
	
	We finish this section with the characterization of the interpolation spaces which follows by \autoref{lem:representation} and the results in \cite{BE:25}.


	\begin{lem}\label{lem:interpolation spaces similarity}			
Let $\Ye := \mathrm{D}(\bB)$ with the graph norm 
		and set $\Yn := \Xn$. 
			For  $\theta \in (0,1)$ and  $r \in (1,\infty)$ the real interpolation spaces $\bY_{\theta,r} = (\Yn,\Ye)_{\theta,r}$ are given by
			\begin{equation*}
				\bY_{\theta,r}
				= 
				\left\{
				\begin{aligned}
				\rH^{1,q}(\Omegap) \times \rB^{2\theta}_{q,r,\Gammap}(\Omegap) \times \rB^{2\theta}_{q,r,\Gammap}(\Omegap) \times \bigl( \rB^{2\theta}_{q,r,\Gammaf}(\Omegaf) \cap \rL^q_{\sigma}(\Omegaf) \bigr) ,
				&\qquad \text{ if } \theta \in (-1,\tfrac{1}{2}+\tfrac{1}{2q}), \\
				\left\{ \begin{aligned} 
					(\up,\vp,\pp,\uf)^\top \in \rH^{1,q}(\Omegap) \times \rB^{2\theta}_{q,r,\Gammap}(\Omegap) \times \rB^{2\theta}_{q,r,\Gammap}(\Omegap)\qquad \phantom{a} \\  \times \bigl( \rB^{2\theta}_{q,r,\Gammaf}(\Omegaf) \cap \rL^q_{\sigma}(\Omegaf) \bigr)
				\colon (\up,\vp,\pp,\uf) \text{ satisfy } \eqref{eq:bjs'}
				\end{aligned} \right\},
				&\qquad \text{ if } \theta \in (\tfrac{1}{2}+\tfrac{1}{2q},1).
				\end{aligned} 
				\right. 
			\end{equation*}
	\end{lem}

	We conclude with the following result for the interpolation spaces of $\bB$.
	
		\begin{cor}\label{lem:interpolation spaces}			
		Let $\Xe := \mathrm{D}({\bA})$ with the graph norm.
		For  $\theta \in (0,1)$ and  $r \in (1,\infty)$ the real interpolation spaces $\bX_{\theta,r} = (\Xe,\Xn)_{\theta,r}$ are given by
			\begin{equation*}
			\bX_{\theta,r}
			= 
			\left\{
			\begin{aligned}
				\rH^{1,q}(\Omegap) \times 	\rB^{2\theta}_{q,r,\Gammap}(\Omegap) \times \rB^{2\theta}_{q,r,\Gammap}(\Omegap) \times \bigl( \rB^{2\theta}_{q,r,\Gammaf}(\Omegaf) \cap \rL^q_{\sigma}(\Omegaf) \bigr) ,  
				&\qquad \text{ if } \theta \in 	(0,\tfrac{1}{2}), \\
				\rH^{1,q}(\Omegap) \times 
				\rB^{1}_{q,r,\Gammap}(\Omegap) \times 	\rB^{1}_{q,r,\Gammap}(\Omegap) \times \bigl( \rB^{1}_{q,r,\Gammaf}(\Omegaf) \cap \rL^q_{\sigma}(\Omegaf) \bigr), &\qquad \text{ if } \theta = \frac{1}{2}, r \leq q, \\
				\rH^{1,q}(\Omegap) \times 
				\rH^{1,q}_{\Gammap}(\Omegap) \times 	\rB^{1}_{q,r,\Gammap}(\Omegap) \times \bigl( \rB^{1}_{q,r,\Gammaf}(\Omegaf) \cap \rL^q_{\sigma}(\Omegaf) \bigr), &\qquad \text{ if } \theta = \frac{1}{2}, q\leq r, \\
				\rH^{1,q}(\Omegap) \times 	\rH^{1,q}_{\Gammap}(\Omegap) \times \rB^{2\theta}_{q,r,\Gammap}(\Omegap) \times \bigl( \rB^{2\theta}_{q,r,\Gammaf}(\Omegaf) \cap \rL^q_{\sigma}(\Omegaf) \bigr) ,
				&\qquad \text{ if } \theta \in 	(\tfrac{1}{2},\tfrac{1}{2}+\tfrac{1}{2q}), \\
				\left\{ \begin{aligned} 
					(\up,\vp,\pp,\uf)^\top \in \rH^{1,q}(\Omegap) \times \rH^{1,q}_{\Gammap}(\Omegap) 	\times \rB^{2\theta}_{q,r,\Gammap}(\Omegap)\qquad \phantom{a} \\  \times \bigl( \rB^{2\theta}_{q,r,\Gammaf}(\Omegaf) \cap \rL^q_{\sigma}(\Omegaf) \bigr)
					\colon \\ \up+\delta \vp \in 	\rH^{2\theta}_{q,r,\Gammap}(\Omegap) \text{ and }
					(\up,\vp,\pp,\uf) \text{ satisfy } 	\eqref{eq:bjs}
				\end{aligned} \right\}, 
				&\qquad \text{ if } \theta \in 	(\tfrac{1}{2}+\tfrac{1}{2q},1).
			\end{aligned} 
			\right. 
	\end{equation*}
	\end{cor}

	\section{Spectral theory}
	\label{sec:spectral theory}
	
	In this section we deal with the spectral theory of the operator $\bA$.
	Due to the (damped) hyperbolic nature of the Biot equation the resolvent of $\bA$ cannot be compact and hence it may has essential spectrum. 
	The main difficulty in this section is to locate the essential spectrum. 
	This part is divided in several steps. 
	
	\medskip 
	
	We point out that the spectra of $\bA$ and $\bB$ coincide. Our aim is to show that the spectral bound $s(\bB)$ of $\bB$ and hence $s({\bA})$ of ${\bA}$ is strictly negative. 
	
	\begin{thm}\label{thm:spectrum A}
		There exists a constant $\omega < 0$ such that $\sigma(\bA) \subset \{ \lambda \in \C \colon \re \lambda \leq \omega \}$.
		In particular, the spectral bound of $\bA$ is strictly negative, i.e. $s(\bA) < 0$. 
	\end{thm}
	
	In the first step we locate the spectrum of the damped wave type part $\Gn$ subject to homogeneous boundary conditions inspired by \cite{Eng:89}. 
	
	\begin{lem}\label{lem:G0 spectrum}
		There exists a constant $\omega < 0$ such that $\sigma(\Gn) \subset \{ \lambda \in \C \colon \re \lambda \leq \omega \}$.
		In particular, the spectral bound of $\Gn$ is strictly negative, i.e. $s(\Gn) < 0$. 
	\end{lem}
	\begin{proof}
		As Lam\'{e} operator $\Ln$ is parameter elliptic and has homogeneous boundary conditions it admits bounded $\mathcal{H}^{\infty}$-calculus by the results in \cite{DDHPV:04}.
		Therefore, we may define $(-\Ln)^{-\frac{1}{2}}$ and note that there exists a constant $\omega_1 > 0$ such that 
		$\sigma((-\Ln)^{-\frac{1}{2}}) \subset \{ \lambda \in \C \colon \re \lambda > \omega_1 \}$. 
		Using the isomorphism
		\begin{equation*}
			S := \begin{pmatrix}
				(-\Ln)^{\frac{1}{2}} & 0 \\
				0 & 1
			\end{pmatrix}
			\colon \rH^{1,q}_{\Gammap}(\Omegap) \times \rL^q(\Omegap)
			\to \rL^q(\Omegap) \times \rL^q(\Omegap),
		\end{equation*}
		consider
		\begin{align*} 
			C := S\Gn S^{-1}
			= 
			\begin{pmatrix}
				-\delta^{-1} & (-\Ln)^{\frac{1}{2}} \\
				-\delta^{-2}(-\Ln)^{-\frac{1}{2}} & \delta \Ln+\delta^{-1} 
			\end{pmatrix}.
		\end{align*}
			Then,	a direct calculation shows that the domain $\mathrm{D}(C) = \rL^q(\Omegap) \times \mathrm{D}(\Ln)$.

		We write $C = M((-\Ln)^{\frac{1}{2}})$ with the matrix
		\begin{equation*}
			M(\alpha) := 
			\begin{pmatrix}
				-\delta^{-1} & \alpha \\
				- \delta^{-2} \alpha^{-1} & -\delta \alpha^2 + \delta^{-1}
			\end{pmatrix} .
		\end{equation*}
		Now, we consider $\mu_0 \in \C$ such that $\mu_0^2 + (1+\delta \mu_0)\alpha^2 \not = 0$, and obtain 
		\begin{equation*}
			(\mu_0 - M(\alpha))^{-1}
			= \frac{1}{\mu_0^2 + (1+\delta \mu_0)\alpha^2} \cdot
			\begin{pmatrix}
				\mu_0 +\delta \alpha^2 - \delta^{-2} & \alpha \\
				-  \delta^{-2} \alpha^{-1} & \mu_0 + \delta^{-1}
			\end{pmatrix} .
		\end{equation*}
		Passing to the limit yields
		\begin{equation*}
			T := \lim_{\alpha \to \infty} 
			(\mu_0 - M(\alpha))^{-1}
			= \begin{pmatrix}
				\frac{\delta}{1+\delta\mu_0} & 0 \\
				0 & 0 
			\end{pmatrix} 
		\end{equation*}
		and therefore $\sigma(T) = \left\{ 0, \frac{\delta}{1+\delta \mu_0}\right\}$. 
		Therefore, the set $\sigma_2(C)$ is given by
		\begin{equation*}
			\sigma_2(C)
			:= \{ \mu_0 - \lambda^{-1} \colon \lambda \in \sigma(T) \setminus \{0\} \}
			= \left\{ - \frac{1}{\delta} \right\}. 
		\end{equation*}
		Consider $\mu \in \sigma((-\Ln)^{\frac{1}{2}}) \subset \{ \lambda \in \C \colon \re \lambda > \omega_1 \}$ a direct estimate shows that there exists a constant $\omega_2 < 0$ such that
		\begin{equation*}
			\frac{1}{2} \re \left( - \delta \mu^2 \pm \sqrt{\delta^2 \mu^4 - 4 \alpha^2} \right) \leq \omega_2 < 0 .
		\end{equation*}
		and we conclude
		\begin{equation*}
			\sigma_1(C)
			:= \bigcup_{\mu \in \sigma((-\Ln)^{\frac{1}{2}})} \sigma(M(\mu))
			\subset \{ \lambda \in \C \colon \re \lambda < \omega_2 \} .
		\end{equation*}
		Finally, \cite[Theorem 2.8]{Eng:89} implies that for $\omega := \max\left\{ \omega_2, - \frac{1}{\delta} \right\} < 0$, we have
		\begin{equation*}
			\sigma(\Gn)
			= \sigma(C)
			= \sigma_1(C) \cup \sigma_2(C)
			\subset \{ \lambda \in \C \colon \re \lambda < \omega < 0 \}
		\end{equation*}
		and the claim follows. 
	\end{proof}
	
	Since $\mathrm{D}(\Deltan) \stackrel{c}{\hookrightarrow} \rL^q(\Omegap)$ and $\mathrm{D}(\An)\stackrel{c}{\hookrightarrow} \rL^q_{\sigma}(\Omegaf)$, both operators have compact resolvents on $\rL^q(\Omegap)$ and $\rL^q_{\sigma}(\Omegaf)$, respectively. So, their spectra only consists of discrete eigenvalues and by the regularity of the eigenfunctions and Sobolev embeddings we see that their spectra are independent of $q \in (1,+\infty)$. On $\rL^2(\Omegap)$ and $\rL^2_{\sigma}(\Omegaf)$, respectively, the operators $\Deltan$ and $\An$ are self-adjoint and hence $\sigma(\Deltan), \sigma(\An) \subset \R$. 
	Integration by parts shows for $\lambda \pp = k \Deltan \pp$ and $\tilde{\lambda} \uf = \An \uf$ that
	\begin{equation*}
		\lambda \| \pp \|_{\rL^2} = - \| \nabla \pp \|_{\rL^2} \leq 0 \qquad \text{ and } \qquad
		\tilde{\lambda} \| \uf \|_{\rL^2} = - \| \nabla \uf \|_{\rL^2} \leq 0 ,
	\end{equation*}
	i.e. $\sigma(\Deltan), \sigma(\An) \subset \{\lambda \in \R \colon \lambda \leq 0\}$. Finally, the inequalities above imply for $\lambda =0$ that $\pp$ and $\uf$ must be constant. Since $\pp|_{\Gammap} = 0$ and $\uf|_{\Gammaf} = 0$ we conclude $\pp = 0$ and $\uf = 0$ and therefore  $\sigma(\Deltan), \sigma(\An) \subset \{\lambda \in \R \colon \lambda < 0\}$. Since, the spectra are discrete there exists a constant $\tilde{\omega} < 0$ such that 
	$\sigma(\Deltan), \sigma(\An) \subset \{\lambda \in \R \colon \lambda < \tilde\omega\}$. Together, with \autoref{lem:G0 spectrum} we obtain that the spectrum of 
	\begin{equation*}
		\bDn
		\coloneqq 
		\begin{pmatrix}
			\Gn & 0 & 0 \\
			0 & k \Deltan & 0 \\
			0 & 0 & \An 
		\end{pmatrix}
		\quad \text{ with domain }
		\quad 
		D(\bDn) = D(\Gn) \times D(\Deltan) \times D(\An) 
	\end{equation*}
	satisfies 
	$\sigma(\bDn) \subset \sigma(\Gn) \cup \sigma(k \Deltan) \cup \sigma(\An) \subset  \{\lambda \in \C \colon \re \lambda \leq \bar\omega\}$ with $\bar{\omega} := \max\{\tilde{\omega},\omega\} < 0$. We recall that 
	\begin{equation}
		\bBn = \bDn + \bQ, \label{eq:splitting}
	\end{equation}
	and we split the perturbation into its lower and upper triangular parts
	\begin{equation*}
		\bQ
		= 
		\bQ_l + \bQ_r 
		\coloneqq  
		\begin{pmatrix}
			0 & 0 & 0 \\
			(\delta^{-1}\alpha \div, -\alpha \div) & 0 & 0 \\
			(0,0) & 0 & 0 
		\end{pmatrix}
		+
		\begin{pmatrix}
			0 & \binom{0}{-\alpha \nabla} & 0 \\
			0 & 0 & 0 \\
			0 & 0 & 0 
		\end{pmatrix}
	\end{equation*}

	In the next step we consider the operator matrix
	\begin{equation*}
		\mathbf{C}_0
		=
		\bDn + \bQ_l, \qquad \text{ with domain } \qquad \mathrm{D}(\mathbf{C}_0) = \mathrm{D}(\bDn).
	\end{equation*} 
	We obtain the following result from the lower triangular structure of $\mathbf{C}_0$. 
	\begin{lem}\label{lem:A00 spectrum}
		There exists a constant $\omega < 0$ such that $\sigma(\mathbf{C}_0) \subset \{ \lambda \in \C \colon \re \lambda \leq \omega \}$. In particular, its spectral bound is strictly negative, i.e. $s(\mathbf{C}_0) < 0$. 
	\end{lem}
	\begin{proof}
		If $\lambda \in \sigma(\Gn)$ we obtain
		\begin{equation*}
			\lambda - \mathbf{C}_0
			= \begin{pmatrix}
				1 & 0 & 0 \\
				(\delta \alpha \div,-\alpha \div)R(\lambda,\Gn) & 1 & 0 \\
				0 & 0 & 1
			\end{pmatrix}
			\cdot 
			\begin{pmatrix}
				\lambda-\Gn & 0 & 0 \\
				0 & \lambda-k\Deltan & 0 \\
				0 & 0 & \lambda-\An 
			\end{pmatrix} .
		\end{equation*}
		Since the first operator on the right hand side is an isomorphism we obtain that  
		\begin{equation*}
			\sigma(\mathbf{C}_0)
			= \sigma(\bDn)
			\subseteq \sigma(\Gn) \cup \sigma (k\Deltan) \cup \sigma(\An) \subset \{\lambda \in \C \colon \re\lambda < \omega\}
		\end{equation*} 
		with $\omega < 0$.
	\end{proof}
	
	Now, we analyze the spectrum of $\bBn$. 
	
	\begin{lem}\label{lem:A0 spectrum}
		There exists a constant $\omega < 0$ such that $\sigma(\bBn) \subset \{ \lambda \in \C \colon \re \lambda \leq \omega \}$. In particular, its spectral bound is strictly negative, i.e. $s(\bBn) < 0$. 
	\end{lem}
	\begin{proof}
		The spectrum can be decomposed into
		$\sigma(\bBn) = \sigma_{\mathrm{ess}}(\bBn) \cup \sigma_d(\bBn) .$
		We first treat the essential spectrum. 
		To this end we recall that 
		\begin{equation*}
			\bBn = \mathbf{C}_0 + \bQ_r.
		\end{equation*}
		From the compact embedding 
		\begin{equation*}
			\mathrm{D}(k\Deltan) \subseteq \rH^{2,q}(\Omegap) \stackrel{c}{\hookrightarrow} \rH^{1,q}(\Omegap)
		\end{equation*}
		and the boundedness of $ \nabla \colon \rH^{1,q}(\Omegap) \to \rL^q(\Omegap)$ we obtain that $\bQ_1$ is relatively $\mathbf{C}_0$-compact. Hence, we conclude from \cite[Theorem 5.35]{Kat:95} and \autoref{lem:A00 spectrum} that
		\begin{equation*}
			\sigma_{\mathrm{ess}}(\bBn)
			= 
			\sigma_{\mathrm{ess}}(\mathbf{C}_0)
			\subseteq 
			\sigma(\mathbf{C}_0)
			\subset \{ \lambda \in \C \colon \re \lambda \leq \omega \} .
		\end{equation*}
		For the point spectrum we note that it is independent of $q$ and hence it suffices to consider the Hilbert space case $q = 2$.
		We equip the space $\Xn$ with
		the inner-product 
		\begin{align*}
			\left\langle({u}_1,{v}_1,p_1,{w}_1),({u}_2,{v}_2,p_2,{w}_2)\right\rangle_{\Xn}:&=2\mup\int\limits_{\Omegap} \varepsilon({u}_1):\varepsilon(\overline{{u}_2}) + \lambdap\int\limits_{\Omegap} \mathrm{div} {u}_1\mathrm{div}\overline{{u}_2} \\
			&+ \int\limits_{\Omegap} {v}_1\cdot \overline{{v}_2}+ \int\limits_{\Omegap} p_1\overline{p_2}+\int\limits_{\Omegaf} {w}_1\cdot \overline{{w}_2}.
		\end{align*}
		and note that $\Xn$ is a Hilbert space. As the spectra of $\bAn$ and $\bBn$ coincide, we consider the equations
		\begin{equation*}
			\left\{
			\begin{aligned}
				\lambda \up &= \vp &&\text{ on } \Omegap, \\
				\lambda \vp &= \Ln \up + \delta \mc{L}_m \vp -\alpha\nabla \pp &&\text{ on } \Omegap, \\
				\lambda \pp &= - \alpha \operatorname{div} \vp + k\Deltan  \pp&&\text{ on } \Omegap, \\
				\lambda \uf &= \An \uf &&\text{ on } \Omegaf,
			\end{aligned}
			\right. 
		\end{equation*}
		subject to the interface conditions \eqref{eq:ns-visco3} and the boundary conditions \eqref{eq:ns-visco4}. Integration by parts and the interface conditions yield
		\begin{align*}
			&2 \mathrm{Re} \lambda \cdot (\| \up \|_{\rH^1}^2+ \| \vp \|_{\rL^2}^2 + \| \pp \|_{\rL^2}^2 + \| \uf \|_{\rL^2}^2) \\ &= (\lambda+\overline{\lambda})\left(2\mup\int\limits_{\Omegap} \varepsilon(\up):\varepsilon(\overline{\up}) + \lambdap\int\limits_{\Omegap} \D \up\D\overline{\up}\right) +  \int_{\Omegap} \lambda \vp \overline{\vp} + \int_{\Omegap} \overline{\lambda} \vp \overline{\vp}+ \int_{\Omegap} \lambda \pp \overline{\pp} +\int_{\Omegap} \pp \overline{\lambda} \overline{\pp} \\&+ \int_{\Omegaf} \lambda \uf \overline{\uf} + \int_{\Omegaf} \uf \overline{\lambda} \overline{\uf} \\
			&= 2\mup\int\limits_{\Omegap} \varepsilon({v}_p):\varepsilon(\overline{\up}) + \lambdap\int\limits_{\Omegap} \D {v}_p\D\overline{\up}+ 2\mup\int\limits_{\Omegap} \varepsilon(\up):\varepsilon(\overline{{v}_p}) + \lambdap\int\limits_{\Omegap} \D \up\D\overline{{v}_p}\\ &+ \int_{\Omegap} (\mc{L}_m \up + \delta \mc{L}_m \vp -\alpha\nabla \pp)\cdot  \overline{\vp} + \int_{\Omegap} (\mc{L}_m \overline{\up} + \delta \mc{L}_m \overline{\vp} -\alpha\nabla \overline{\pp})\cdot  {\vp} +\int_{\Omegap} (- \alpha \operatorname{div} \vp + k\Delta_m  \pp) \overline{\pp} \\ &+\int_{\Omegap} (- \alpha \operatorname{div} \overline{\vp} + k\Delta_m  \overline{\pp}) {\pp} + \int_{\Omegaf} A_m u_{f} \overline{u_{f}} + \int_{\Omegaf} u_{f} A_m \overline{\uf} \\
			&= -2\delta\mup \| \varepsilon(\vp)\|_{\rL^2}^2-\delta\lambdap \|\D \vp \|_{\rL^2}^2  -2k \| \nabla \pp \|_{\rL^2}^2  -2\mu_f \| \varepsilon(\uf)\|_{\rL^2}^2
			\leq 0 .
		\end{align*}
		Hence, $\mathrm{Re}\lambda \leq 0$. Finally, for the case $\mathrm{Re}\lambda = 0$ we obtain that $\vp$, $\pp$ and $\uf$ are constant. Now the boundary conditions \eqref{eq:ns-visco4} imply $\vp = \pp = 0$ in $\Omegap$ and $\uf = 0$ in $\Omegaf$. Trivially, we have $\up=0$ in $\Omegap$. Thus, we can conclude that for some constant $\omega<0$, $
		\sigma_{{d}}(\bBn)
		\subset \{ \lambda \in \C \colon \re \lambda \leq \omega \} .
		$
	\end{proof}
		
	With our preparation from \autoref{sec:representation} and the control of the spectrum of the operator $\bBn$ with homogeneous boundary conditions, we are	
	now sufficient prepared to prove \autoref{thm:spectrum A}.
	
	\begin{proof}[{Proof of \autoref{thm:spectrum A}.}]
		We know $\sigma(\bA) = \sigma(\bB)$ and
		\begin{equation*}
			\sigma(\bB)
			\subseteq \sigma(\bBn) \cup (\sigma(\bB) \cap \rho(\bBn)). 
		\end{equation*}
		Moreover, by using \autoref{lem:A0 spectrum} we have $\sigma(\bBn) \subset \{ \lambda \in \C \colon \re \lambda \leq \omega \}$. Thus, it is enough to show: there exists a constant $\omega < 0$ such that $(\sigma(\bB) \cap \rho(\bBn)) \subset \{ \lambda \in \C \colon \re \lambda \leq \omega \}$. In the sequel we assume $\lambda~\in~\rho(\bBn)$. 
		Now, we use the fact that the boundary perturbation $\Phi$ is of lower order. 
	From \eqref{eq:formula Llambda} we obtain
		\begin{equation}
			\Phi \mathbf{L}^{\bB}_\lambda 
			= \Phi \mathbf{L}^{\bD}_\lambda 
			+ \Phi R(\lambda,\bBn) \bQ \mathbf{L}^{\bD}_\lambda .
			\label{eq:phiLlambda}
		\end{equation}
		Using
		the mapping properties of $\text{L}_\lambda$ from \autoref{lem:L0} and the trace theorem we obtain 
		\begin{align*}
			\Phi \mathbf{L}^{\bD}_\lambda
			\colon \rW^{1-\frac{1}{q},q}(\Gamma) \times\rW^{1-\frac{1}{q},q}(\Gamma)
			\times \rW^{1-\frac{1}{q},q}(\Gamma)
			&\to \rW^{2-\frac{1}{q},q}(\Gamma) \times \rW^{2-\frac{1}{q},q}(\Gamma) \times \rW^{2-\frac{1}{q},q}(\Gamma) \\
			&\stackrel{c}{\hookrightarrow}
			\rW^{1-\frac{1}{q},q}(\Gamma) \times \rW^{1-\frac{1}{q},q}(\Gamma) \times \rW^{1-\frac{1}{q},q}(\Gamma) 
		\end{align*}
		and the first term on the right-hand side of \eqref{eq:phiLlambda} is compact on $\rW^{1-\frac{1}{q},q}(\Gamma) \times\rW^{1-\frac{1}{q},q}(\Gamma) \times \rW^{1-\frac{1}{q},q}(\Gamma)$.
		For the second term we first note that $\mathrm{D}(\bBn) = \mathrm{D}(\Gn) \times \mathrm{D}(\Deltan) \times \mathrm{D}(\An) \hookrightarrow \rH^{1,q}(\Omegap) \times  \rH^{1,q}(\Omegap) \times  \rH^{1,q}(\Omegap) \times  \rH^{1,q}(\Omegaf)$ and by the trace theorem  
		\begin{equation*} 
		\Phi R(\lambda,\bBn) \colon \rH^{1,q}(\Omegap) \times \rL^{q}(\Omegap) \times \rL^{q}(\Omegap) \times \rL^{q}_{\sigma}(\Omegaf) \to 
		\rW^{1-\frac{1}{q},q}(\Gamma) \times \rW^{1-\frac{1}{q},q}(\Gamma) \times \rW^{1-\frac{1}{q},q}(\Gamma) 
		\end{equation*} 
		is bounded. Using the structure of the perturbation $\bB$ and the mapping properties of $\mathbf{L}^{\bD}_\lambda$ from \autoref{lem:L0} we obtain
		\begin{align*}
			\bQ \mathbf{L}^{\bD}_\lambda
			\colon \rW^{1-\frac{1}{q},q}(\Gamma) \times\rW^{1-\frac{1}{q},q}(\Gamma) \times \rW^{1-\frac{1}{q},q}(\Gamma) 
			&\to \{0\} \times \rH^{1,q}(\Omega_p)
			\times \rH^{1,q}(\Omega_p)
			\times \{0\} \\
			&\stackrel{c}{\hookrightarrow}
			\rH^{1,q}(\Omegap) \times \rL^{q}(\Omegap) \times \rL^{q}(\Omegap) \times \rL^{q}_{\sigma}(\Omegaf) 
		\end{align*}
		and hence $\bQ \mathbf{L}^{\bD}_\lambda$ is compact. It follows that
		$\Phi R(\lambda,\bBn)\bQ \mathbf{L}^{\bD}_\lambda$ is a compact operator on the boundary space $\rW^{1-\frac{1}{q},q}(\Gamma) \times\rW^{1-\frac{1}{q},q}(\Gamma) \times \rW^{1-\frac{1}{q},q}(\Gamma)$. 
		Finally, we conclude from \eqref{eq:phiLlambda} that $\mathbf{L}^{\bB}_\lambda$ is a compact operator on $\rW^{1-\frac{1}{q},q}(\Gamma) \times\rW^{1-\frac{1}{q},q}(\Gamma) \times \rW^{1-\frac{1}{q},q}(\Gamma)$.
		Now, it follows from \autoref{lem:representation} and \cite[Corollary 3.13(b)]{AE:18} that 
		\begin{equation*}
			(\sigma(\bB) \cap \rho(\bBn)) 
			= (\sigma_p(\bB) \cap \rho(\bBn)) .
		\end{equation*}
		As in the steps of the second part of the proof of \autoref{lem:A0 spectrum}, using the 
		integration by parts and the interface conditions, we obtain
		that there exists a constant $\omega < 0$ such that
		$\sigma_p(\bB) \subset \{ \lambda \in \C \colon \re \lambda \leq \omega \}$ and the claim follows. 
	\end{proof}

	\section{Maximal regularity}
	\label{sec:mr}
	
	In this section we address the maximal regularity of $\bA$. First we prove maximal regularity of $\Gn$, then by perturbation theory we obtain maximal regularity of ${\bA}_0$. Further, we use a boundary perturbation argument in order to obtain the maximal regularity of ${\bA}$. Finally, similarity transform yield the desired result.
	
	\begin{thm}\label{thm:mr}
		The operator $\bA$ admits maximal $\rL^r$-regularity on $\Xn$ for all $r \in (1,\infty)$. 
	\end{thm} 
	
	Again, we start with the maximal regularity of the damped wave type part.
	
	\begin{lem}\label{lem:G0 MR}
		The operator $\Gn$ admits a bounded $\mathcal{H}^\infty$-calculus on $\rH^{1,q}(\Omegap) \times \rL^q(\Omegap)$. In particular, it has maximal $\rL^p$-regularity on $\rH^{1,q}(\Omegap) \times \rL^q(\Omegap)$.
	\end{lem}
	\begin{proof}
		First, note that in consideration of \autoref{lem:G0 spectrum} is suffices to show that $\Gn+\omega$ admits maximal $\rL^p$-regularity for a large constant $\omega \in \R$.
		
		Since $\Ln$ is parameter-elliptic, it admits a bounded $\mathcal{H}^\infty$-calculus on $\rL^q(\Omegap)$ by \cite{DDHPV:04}. 
		In order to get make the space isotropic, we use again the isomorphism
		\begin{equation*}
			S := \begin{pmatrix}
				(-\Ln)^{\frac{1}{2}} & 0 \\
				0 & 1
			\end{pmatrix}
			\colon \rH^{1,q}_{\Gammap}(\Omegap) \times \rL^q(\Omegap)
			\to \rL^q(\Omegap) \times \rL^q(\Omegap)
		\end{equation*}
		and obtain as above
		\begin{align*} 
			C := S\Gn S^{-1}
			= 
			\begin{pmatrix}
				-\delta^{-1} & (-\Ln)^{\frac{1}{2}} \\
				-\delta^{-2}(-\Ln)^{-\frac{1}{2}} & \delta \Ln+\delta^{-1} 
			\end{pmatrix}
			\qquad \text{ with domain }
			\qquad \mathrm{D}(C) = \rL^q(\Omegap) \times \mathrm{D}(\Ln) .
		\end{align*}
		We split the operator $C$ as
		\begin{equation*}
			C
			= 	
			\begin{pmatrix}
				0 & 0 \\
				0 & \delta \Ln
			\end{pmatrix}
			+
			\begin{pmatrix}
				0 & (-\Ln)^{\frac{1}{2}} \\
				0 & 0
			\end{pmatrix}
			+ 
			\begin{pmatrix}
				-\delta^{-1} & 0 \\
				-\delta^{-2}(-\Ln)^{-\frac{1}{2}} & \delta^{-1} 
			\end{pmatrix}
			=: D + Q_1 + Q_0 .
		\end{equation*}
		Note that $Q_0$ is a bounded operator and $Q_1$ is relatively $C^{\frac{1}{2}}$-bounded. 
		Now the claim follows from \cite{DDHPV:04}.
	\end{proof}
	In the next step, we analyze the maximal regularity of $\bBn$. 
	
	\begin{lem}\label{lem:A0 mr}
		The operator $\bBn$ admits maximal $\rL^p$-regularity on $\Xn$ for all $p \in (1,\infty)$.
	\end{lem}  
	\begin{proof}
		First, note that in consideration of \autoref{lem:A0 spectrum} is suffices to show that $\bBn + \omega$ admits maximal $\rL^p$-regularity for a large constant $\omega \in \R$. 			
		We recall the splitting
		\begin{equation*}
			\bBn = \bDn + \bQ 
		\end{equation*}
		from \eqref{eq:splitting}. 
		Using \autoref{lem:G0 MR} it follows that the first operator matrix on the right hand side admits maximal $\rL^p$-regularity on $\Xn$.
		Since $(\delta \alpha \div,-\alpha \div)$ is relatively $\Gn$-bounded and $-\alpha \nabla$ is relatively $k\Deltan$-bounded with bound $0$, \cite{DDHPV:04} implies the claim. 
	\end{proof}
	Now, we are able to prove \autoref{thm:mr}.
	
	\begin{proof}
		First, note that in consideration of \autoref{thm:spectrum A} is suffices to show that $\bA +\omega$ admits maximal $\rL^p$-regularity for a large constant $\omega \in \R$. Moreover, note that by similarity transform it suffices to show maximal $\rL^p$-regularity of $\bB+\omega$ for for a large constant $\omega \in \R$. We recall from \autoref{lem:representation} that
		\begin{equation*}
			\bB
			= (\bB_{-1} + (\lambda-\bB_{-1}) \mathbf{L}_\lambda^{\bB} \Phi)|_{\Xn}. 
		\end{equation*}		
		By the mapping properties of $\mathbf{L}_\lambda^{\bB}$ from \autoref{lem:L0_real} and the trace theorem we obtain 
		\begin{equation*}
			\mathbf{L}^{\bB}_\lambda \Phi \in \mathcal{L}(\bZ_{\frac{1}{2}},\bZ_{\beta})
		\end{equation*}
		for $\beta < \frac{1}{2}\bigl(1+\frac{1}{q}\bigr)$, where $\bZ_\theta$ denote the interpolation spaces from \autoref{lem:interpolation spaces A0}.
		This implies 
		\begin{equation*}
			(\lambda-\bB_{-1}) \mathbf{L}^{\bB}_\lambda \Phi \in \mathcal{L}(\bZ_{\frac{1}{2}},\bZ_{\beta-1})
		\end{equation*}
		and the result follows from \cite[Proposition 3.4]{BHR:24}.
	\end{proof}


	\section{Strong wellposedness}
	\label{sec:wellposedness}
	
	We recall from \autoref{sec:visco} that the system \eqref{eq:biot}--\eqref{eq:nse}--\eqref{eq:bjs} supplemented to \eqref{eq:bdry} and \eqref{eq:initial} can be reformulated as a semilinear evolution equation on the ground space $\Xn$
	\begin{equation}\label{eq:aCP}
		\frac{d}{dt} \mathbf{w}
		=\
		\bA \mathbf{w}+ \bF(\mathbf{w},\mathbf{w}) ,
		\quad  		
		\mathbf{w}(0)=( {u}_0, {v}_0, p_0, {u}^{f}_0)^\top,
	\end{equation}
	with $\mathbf{w}:= (\up,\vp,\pp,\uf)^\top$, where  $\bA$ is as above and the nonlinear term $\bF(\bw',\bw)
	= (0,0,0, -\mathcal{P} ({u'}_f \cdot \nabla \uf))^\top$ as in Section \ref{sec:visco}. Furthermore, we consider the maximal regularity spaces 
	\begin{align*}
		\mathbb{E}_{1,\mu}(T) &:= \rH^{1,r}_{\mu}(0,T;\Xn) \cap  \rL^r_{\mu}(0,T;\mathbf{X}_1) \mbox{ and } \\
		_0 \mathbb{E}_{1,\mu}(T) &:= \{{w}\in\rH^{1,r}_{\mu}(0,T;\Xn) \cap  \rL^r_{\mu}(0,T;\mathbf{X}_1) \mid {w}(0) = 0 \}, 
	\end{align*}
	the data space $\E_{0,\mu}(T) := \rL^r_\mu(0,T;\Xn)$ as well as the trace space ${X}_{\gamma,\mu} := [\Xn,\mathbf{X}_1]_{\mu-\frac{1}{r},r}$. Furthermore, we use the spaces
	\begin{align*}
		\tilde{\mathbb{E}}_{1,\mu}^f(T) &:= \rH^{1,r}_{\mu}(0,T;\rL^q(\Omegaf)) \cap  \rL^r_{\mu}(0,T;\rH^{2,q}(\Omegaf)) \mbox{ and } \\
		_0 \tilde{\mathbb{E}}^f_{1,\mu}(T) &:= \{{u^f}\in\rH^{1,r}_{\mu}(0,T;\rL^q(\Omegaf)) \cap  \rL^r_{\mu}(0,T;\rH^{2,q}(\Omegaf)) \mid {u^f}(0) = 0 \}  
	\end{align*}
	and $\tilde{\E}_{0,\mu}^f := \rL^q_{\mu}(0,T;\rL^q(\Omegaf))$.
	 The following estimate for the Navier-Stokes nonlinearity is essentially known, see \cite{PSW:18}. 
	
	\begin{lem}\label{lem:F}
		Let $r, q \in (1,+\infty)$.
		Then the bilinear mapping  
		\begin{equation*} 
			F \colon \tilde{\E}_{1,\mu}^f(T) \times  \tilde{\E}^f_{1,\mu}(T) \to \tilde{\E}_{0,\mu}^f(T) \colon (u_1,u_2) \mapsto u_1 \cdot \nabla u_2
		\end{equation*} 
		satisfies
		\begin{equation*}
			\|F(u_1,u_1)-F(u_2,u_2)\|_{\tilde{\E}^f_{0,\mu}(T)} \leq C \cdot \bigl( \| u_1 \|_{\tilde{\E}^f_{1,\mu}(T)} + \| u_2 \|_{\tilde{\E}^f_{1,\mu}(T)} \bigr) \cdot \| u_1 - u_2 \|_{\tilde{\E}^f_{1,\mu}(T)} 
		\end{equation*}
		for all $u_1, u_2 \in \tilde{\E}^f_{1,\mu}(T)$.
	\end{lem}	
	We now give the proof of our main result concerning local strong wellposedness and global strong wellposedness for small data for the system \eqref{eq:biot}--\eqref{eq:nse}--\eqref{eq:bjs}.

	\begin{proof}[Proof of \autoref{thm:main}.]
		Let $\mu_0 := \frac{1}{2} + \frac{1}{2q}+\frac{1}{r}$ and $\Xn$ and $\Xe$ defined as in \eqref{eq:X0} and \eqref{eq:X1}, respectively.
	It follows from \autoref{lem:interpolation spaces} that 
		for initial data $\bw_0 := ({u}_0, {v}_0, p_0, {u}_0^f)^\top \in \rH^{1,q}_{\Gammap}(\Omegap) \times \rB_{q,r,\Gammap}^{2\mu-\frac{2}{r}}(\Omegap) \times
		\rB_{q,r,\Gammap}^{2\mu-\frac{2}{r}}(\Omegap)\times
		\rB_{q,r,\Gammaf,\sigma}^{2\mu-\frac{2}{r}}(\Omegaf)$ for $\mu > \mu_0$, belong to the interpolation space $X_{\gamma,\mu} = (\Xn,\Xe)_{\mu-\frac{1}{r},r}$. Using \autoref{thm:mr}, we obtain that $w_*(t) := e^{t\bA} \bw_0 \in \E_{1,\mu}(T)$. Now, $\bar{\bw}:= \bw-\bw^\ast$ satisfies
		\begin{equation*}
			\frac{\d}{\d t} \bar{\bw}
			=\
			\bA \bar{\bw} + \bF(\bar{\bw}+\bw^*,\bar{\bw}+\bw^*) ,
			\quad  		
			\bar{\bw}(0)=0 .
		\end{equation*}
		Rewriting \eqref{eq:aCP} as a fixed point equation, we obtain 
		\begin{equation*}
			\bar{\bw} = e^{t\bA} \ast \bF(\bar{\bw}+\bw^*,\bar{\bw}+\bw^*) =: \Phi(\bar{\bw}) .
		\end{equation*}
		Using \autoref{thm:mr} and \autoref{lem:F} we obtain
		\begin{align*}
			\| e^{t\bA} \ast\bF (\bar{\bw}+w^*,\bar{\bw}+\bw^*) \|_{\E_{1,\mu}}
			&\leq C \cdot \| \bF (\bar{\bw}+\bw^*,\bar{\bw}+\bw^*) \|_{\E_{0,\mu}} = C \cdot \| F(\bar{u}+u^\ast,\bar{u}+u^\ast) \|_{\rL^r_{\mu}(0,T;\rL^q_{\sigma}(\Omegaf))} \\
			&\leq C \cdot \| \bar{u} + u^\ast \|_{\tilde{\E}_{1,\mu}^f}^2 
			\leq C \cdot \| \bar{\bw} + \bw^\ast \|_{\E_{1,\mu}}^2 \leq 2 C \cdot \bigl( \| \bar{\bw} \|_{\E_{1,\mu}}^2 + \| \bw^\ast \|_{\E_{1,\mu}}^2  \bigr) .
		\end{align*}
		Consider $B_r(0) \subset { }_0\E_{1,\mu}(T)$. Choosing $T >0$ small we see that $2 C \cdot\| \bw^\ast \|_{\E_{1,\mu}}^2  < \frac{r}{2}$ and $2 C \cdot \| \bar{\bw} \|_{\E_{1,\mu}}^2< \frac{r}{2}$. Hence $\Phi \colon B_r(0) \to B_r(0)$ is a self-map. Similar, we see with \autoref{lem:F} that 
		\begin{equation}
			\| \Phi(\bar{\bw}_1) - \Phi(\bar{\bw}_2) \|_{\E_{1,\mu}}
			\leq C \cdot (\| \bw_1\|_{\E_{1,\mu}}
			+\| \bw_2\|_{\E_{1,\mu}}) \cdot \| \bw_1-\bw_2\|_{\E_{1,\mu}}
			\leq \tilde{C} \cdot  \| \bw_1-\bw_2\|_{\E_{1,\mu}}
		\end{equation}
		with $\tilde{C} < 1$ for $T$ small. Hence, $\Phi \colon B_r(0) \to B_r(0)$ is a contraction and the local existence follows from Banach's fixed point theorem. 
		Further, note that the maximal regularity constant is independent of $T$, since the semigroup is exponentially stable, and hence we may also choose $r$ small instead of $T$. This yields global existence for small initial data. 
	\end{proof}

	Finally, we prove the finite-in-time blow up criterion. 
	
	\begin{proof}[Proof of {\autoref{cor:bjs blow up}}.]
		Consider the trace space ${X}_{\gamma,\mu} := [\Xn,\mathbf{X}_1]_{\mu-\frac{1}{r},r}$. For the critical weight $\mu_c := \frac{3}{2q}-\frac{1}{2}+\frac{1}{r}$, we know from \autoref{lem:interpolation spaces}(i) that $\bX_{\gamma,\mu_c}$ is given by
		\begin{equation}
			\bX_{\gamma,\mu_c}
			= \rH^{1,q}(\Omegap) \times  \rB_{q,r,\Gammap}^{\frac{3}{q}-1}(\Omegap) \times  \rB_{q,r,\Gammap}^{\frac{3}{q}-1}(\Omegap)\times \bigl( \rB^{\frac{3}{q}-1}_{q,r,\Gammaf}(\Omegaf) \cap \rL^q_{\sigma}(\Omegaf) \bigr).
			\label{eq:trace spaces}
		\end{equation}	
		Claim (a) follows immediately from the relation \eqref{eq:embedding mr} and the above characterization \eqref{eq:trace spaces}.  
		
		\medskip 
		
		For Claim (b) we follow the proof of \cite[Theorem 2.4(ii)]{PSW:18}.
		Suppose $t_+ < \infty$. 
		By mixed derivative theorem, see \cite{MS:12}, and Sobolev embedding, we obtain the embeddings
		\begin{align*}
			(\rL^r(0,T;\rH^{\frac{3}{q}-1-\frac{2}{r}}(\Omegaf)\cap \rL^q_{\sigma}(\Omegaf));\tilde{\E}_{1,\mu_c}^f(T))_{\frac{1}{2}}
			&\hookrightarrow 
			(\rL^r(0,T);\rH^{\frac{3}{q}-1-\frac{2}{r}};\rH^{\frac{1}{r},r}_{\mu_c}(0,T;\rH^{2-\frac{2}{r},q}(\Omegaf) \cap \rL^q(\Omegaf)) )_{\frac{1}{2}} \\
			&\hookrightarrow 
			\rH^{\frac{1}{2r},r}_{\frac{1+\mu_c}{2}}(0,T;\rH^{\frac{1}{2}+\frac{3}{2q},q}(\Omegaf) \cap \rL^q_{\sigma}(\Omegaf)) \\
			&\hookrightarrow 
			\rL^{2r}_{\tau_c}(0,T;\rH^{s,q}(\Omegaf) \cap \rL^q_{\sigma}(\Omegaf))
		\end{align*}
		with $s = \frac{1}{2}+\frac{3}{2q}$ and $2(1-\tau_c) = 1-\mu_c$, where $\mu_c := \frac{3}{2q}-\frac{1}{2}+\frac{1}{r}$ is the critical weight. We have 
		\begin{equation*}
			\| F(u,u) \|_{\rL^{r}_{\mu_c}(T_0,T;\rL^q(\Omegaf))}
			\leq C \cdot \| u \|_{\rL^{2r}_{\tau_c}(T_0,T;\rH^{s,q}(\Omegaf))}^2
			\leq C \cdot \| u \|_{\rL^r(T_0,T;\rH^{\frac{3}{q}-1+\frac{2}{r}}(\Omegaf))} \cdot \| u \|_{\tilde{\E}_{1,\mu_c}^f(T_0,T)} .
		\end{equation*}
		for $T_0 < T < t_+$ and the constant $C > 0$ is independent of $T$, where the spaces are the obvious modifications of the definitions at the beginning of this section. 
		Using 
		\begin{equation*} 
		X_{\mu_c} \hookrightarrow
		\rH^{1,q}(\Omegap) \times \rH^{\min\{1,\frac{3}{q}-1+\frac{2}{r}\},q}_{\Gammap}(\Omegap) \times \rH^{\frac{3}{q}-1-\frac{2}{r},q}_{\Gammap}(\Omegap) \times \bigl( \rH^{\frac{3}{q}-1-\frac{2}{r},q}_{\Gammaf}(\Omegaf) \cap \rL^q_{\sigma}(\Omegaf) \bigr)
		\end{equation*}
		this implies
		\begin{equation*}
			\| \bF(w,w) \|_{\rL^r_{\mu_c}(T_0,T;\Xn)}
			\leq C \cdot \| \mathbf{w}\|_{\rL^r(T_0,T;X_{\mu_c})} \cdot \| \mathbf{w}\|_{\E_1(T_0,T)} . 
		\end{equation*}
		 We denote by $M >0$ the maximal regularity constant on the interval $[0,t_+]$,
		and set $\eta := \frac{1}{2CM}$. Now, we choose $T_0$ sufficient close to $t_+$ such that $\| \mathbf{w}\|_{\rL^r(T_0,t_+;X_{\mu_c})} \leq \eta$. 
		Maximal regularity implies the estimate
		\begin{equation*}
			\| \mathbf{w}\|_{\E_{1,\mu_c}(T_0,T)}
			\leq M \cdot \bigl( \| \mathbf{w}(T_0) \|_{X_{\gamma,\mu}} + C \eta \| \mathbf{w}\|_{\E_{1,\mu_c}(T_0,T)} \bigr) .
		\end{equation*}
		An absorbing argument yields
		\begin{equation*}
			\| \mathbf{w}\|_{\E_{1,\mu_c}(T_0,T)} \leq 2 M \| w(T_0) \|_{X_{\gamma,\mu}}
		\end{equation*}
		for any $T \in (T_0,t_+)$. This implies 
		\begin{equation*}
			\bw|_{(T_0,t_+)} \in \E_{1,\mu_c}(T_0,t_+) \hookrightarrow \rC([T_0,t_+],X_{\mu_c}) 
		\end{equation*}
		and hence, the solution $w$ can be continued beyond $t_+$.
		This contradicts the choice of $t_+$. 
	\end{proof}

\medskip 

{\bf Acknowledgements}
T.B would like to thank DFG for support through project 538212014. 
M. H acknowledges the support by DFG project FOR 5528. A.R is supported by the Grant RYC2022-036183-I funded by MICIU/AEI/10.13039/501100011033 and by ESF+. A.R has been partially supported by the Basque Government through the BERC 2022-2025 program and by the Spanish State Research Agency through BCAM Severo Ochoa CEX2021-001142-S and through project PID2023-146764NB-I00 funded by MICIU/AEI/10.13039/501100011033 and cofunded by the European Union.


\begin{thebibliography}{99}
	
	
	
	
	\bibitem{AE:18} M. Adler and K.-J. Engel.
	Spectral theory for structured perturbations of linear operators.
	{\it J. Spectr. Theory} {\bf 8} (2018), 1393--1442, {\url{https://dx.doi.org/10.4171/JST/230}}.
	
	
	\bibitem{Ama:19}
	H.~Amann,
	{\it Linear and quasilinear parabolic problems. Vol. II.} Monographs in Mathematics, vol.~106, Birkhäuser/Springer, Cham, 2019,
	{\url{https://doi.org/10.1007/978-3-030-11763-4}}.
	
	
	\bibitem{AKYZ:18} I. Ambartsumyan, E. Khattatov, I. Yotov and P. Zunino.
	A Lagrange multiplier method for a Stokes-Biot fluid-poroelastic structure interaction model.
	{\it Numer. Math.} {\bf 140} (2018), 513--553, {\url{https://dx.doi.org/10.1007/s00211-018-0967-1}}.

	
	\bibitem{Ang:90a} S.B. Angenent.
	Parabolic equations for curves on surfaces. I. Curves with p-integrable curvature.
	{\it Ann. of Math.} {\bf 132} (1990), 451--483, {\url{https://doi.org/10.2307/1971426}}.

	
	
	\bibitem{JLA:80} J.L. Auriault. Dynamic behaviour of a porous medium saturated by a Newtonian fluid. {\it International Journal of Engineering Science} {\bf 18, no. 6} (1980): 775--785, {\url{https://doi.org/10.1016/0020-7225(80)90025-7}}.
	
	\bibitem{AGW:24} G. Avalos, E. Gurvich and J. Webster.
	Weak and Strong Solutions for A Fluid-Poroelastic-Structure
	Interaction via a Semigroup Approach.
	{\it Mathematical Methods in the Applied Sciences} {\bf 48, no. 4} (2025), 4057--4089, {\url{https://doi.org/10.1002/mma.10533}}.
	
	
	\bibitem{BQQ:09} S. Badia, A. Quaini and A. Quarteroni. 
	Coupling Biot and Navier-Stokes equations for modelling fluid-poroelastic media interaction.
	{\it J. Comput. Phys.} {\bf 228} (2009), 7986--8014, {\url{https://dx.doi.org/10.1016/j.jcp.2009.07.019}}.
	
	
	\bibitem{BJ:67} G. Beavers and D. Joseph. 
	Boundary conditions at a naturally permeable wall.
	{\it J. Fluid. Mech.} {\bf 269} (1967), 197--207. {\url{https://dx.doi.org/10.1017/S0022112067001375}}.
	
	
	
	\bibitem{BE:18} T. Binz and K.-J. Engel. 
	Operators with Wentzell boundary conditions and the Dirichlet-to-
	Neumann operator.
	{\it Math. Nachr.} {\bf 292} (2019), 733--746, {\url{https://dx.doi.org/10.1002/mana.201800064}}.
	
	\bibitem{BE:25} T. Binz and K.-J. Engel. 
	Staffans-Weiß perturbation theory: Interpolation spaces.
	preprint (2025).
	
	
	\bibitem{BHR:24} T. Binz, M. Hieber and A. Roy.
	Fluid-structure interaction with poroelastic material: the Beavers-Joseph condition in the strong sense. 
	{\it J. Differential Equ.} {\bf 426} (2025),  660--689, {\url{https://dx.doi.org/10.1016/j.jde.2025.01.042}}.
	
	
	\bibitem{Bio:41} M. Biot. 
	General theory of three-dimensional consolidation.
	{\it J. Appl. Phys.} {\bf 12} (1941), 155--164, {\url{https://dx.doi.org/10.1063/1.1712886}}.
	
	
	\bibitem{Bio:56} M. Biot. 
	Theory of deformation of a porous viscoelastic anisotropic solid.
	{\it J. Appl. Phys.} {\bf 27} (1956), 459--467, {\url{https://dx.doi.org/10.1063/1.1722402}}.
	
	
	\bibitem{BCMW:21} L. Bociu, S. Cani\'c, B. Muha and J. Webster. 
	Multilayered Poroelasticity Interaction with Stokes Flow.
	{\it SIAM J. Math. Anal.} {\bf 53} (2021), 6243--6279, {\url{https://dx.doi.org/10.1137/20M1382520}}.
	
	\bibitem{BGSW:16} L. Bociu, G. Guidoboni, R. Sacco and J. Webster. 
	Analysis of nonlinear poro-elastic and poro-visco-elastic models.
	{\it Arch. Ration. Mech. Anal.} {\bf 3} (2016), 1445 -- 1519, {\url{https://dx.doi.org/10.1007/s00205-016-1024-9}}.
	
	%
	\bibitem{BMW:23} L. Bociu, B. Muha and J. Webster. 
	Mathematical Effects of Linear Visco-elasticity in Quasi-static Biot Models.
	{\it Journal of Mathematical Analysis and Applications} {\bf 527, no. 2} (2023), 127462, {\url{https://doi.org/10.1016/j.jmaa.2023.127462}}.
	
	
	
	\bibitem{BPY:21} J.W. Both, I.S. Pop and I. Yotov. Global existence of weak solutions to unsaturated poroelasticity. {\it ESAIM: Mathematical Modelling and Numerical Analysis} {\bf 55, no. 6} (2021), 2849--2897, \url{https://doi.org/10.1051/m2an/2021063}.
	
	\bibitem{BCM:25} F. Brandt, S. \v{C}ani\'{c} and B. Muha. Three-dimensional Navier-Stokes-Biot coupling via a moving reticular plate interface: existence of weak solutions. {\it arXiv preprint arXiv:2508.14310} (2025).
	
	\bibitem{brandt2026} F. Brandt, S. {\v{C}}ani{\'c}, A. Scharf and J. Tamba{\v{c}}a. A fully averaged poroelastic Kirchhoff plate interacting with an incompressible, viscous fluid: analysis and numerical simulation. {\it arXiv preprint arXiv:2605.17496} (2026).
	
	
	
	
	\bibitem {BYZ:15}
	M. Bukac, I. Yotov and P. Zunino.
	An operator splitting approach for the interaction between a fluid and a multilayered poroelastic structure. 
	\textit{Numer. Methods PDE} \textbf{31} (2015), 1054--1100,
	\url{https://doi.org/10.1002/num.21936}.
 
	
	
	
	\bibitem{Ces:17} A. Cesmelioglu. 
	Analysis of the coupled Navier-Stokes/Biot problem.
	{\it J. Math. Anal. Appl.} {\bf 456} (2017), 970--991, {\url{https://dx.doi.org/10.1016/j.jmaa.2017.07.037}}.
	
	

	
	
	
	
	
	
\bibitem{DDHPV:04}
	R. Denk, G. Dore, M. Hieber, J. Pr{ü}ss and A. Venni. 
	New thoughts on old results of R.T. Seeley
	{\it Math. Ann.} {\bf{328}} (2004), 545--583,
	{\url{https://doi.org/10.1007/s00208-003-0493-y}}.
	
\bibitem{DHP:04}
	R. Denk, M. Hieber and J. Pr{ü}ss.
	Optimal $\rL^p$-$\rL^q$-estimates for parabolic boundary value problems with inhomogeneous data.
	{\it Math. Z.} {\bf{257}} (2007), 193--224,
	{\url{https://doi.org/10.1007/s00209-007-0120-9}}.
	
	
	\bibitem{Eng:89}
	K.-J. Engel.
	A spectral mapping theorem for polynomial operator matrices.
	{\it Differential Integral Equations} {\bf{2}} (1989), 203--215.
	
	
	

	




	
	
	
	
	
	\bibitem {Gre:87}
	G. Greiner.
	Perturbing the boundary conditions of a generator. 
	{\it Houston J. Math.} {\bf 13} (1987) , 213--229.
	
	
	
	
	\bibitem{JM:00} W. Jäger and A. Mikelic. 
	On the interface boundary condition of Beavers, Joseph and Saffman.
	{\it SIAM J. Appl. Math.} {\bf 60} (2000), 1--25, {\url{https://doi.org/10.1137/S003613999833678X}}.
	

	\bibitem{Kat:95} T. Kato. 
	Perturbation theory for linear operators, {\it Classics in Mathematics} {\bf 132}, {Springer} 1995.
	{\url{https://doi.org/10.1007/978-3-642-66282-9}}.


		
	
	
			\bibitem{KCM:24}
 J. Kuan, S. \v{C}ani\'{c} and B. Muha. Fluid-poroviscoelastic structure interaction problem with nonlinear coupling. \textit{Journal de Mathematiques Pures and Appliques}, {\bf188} (2024), 345--445, \url{https://doi.org/10.1016/j.matpur.2024.06.004}.

	
	\bibitem{KCM:25} J. Kuan, S. \v{C}ani\'{c} and B. Muha. A Regularized Interface Method for Fluid-Poroelastic Structure Interaction Problems with Nonlinear Geometric Coupling. {\it arXiv preprint arXiv:2508.18065} (2025).
	
	

	
	
	
	
	
	
	\bibitem{LM:72}
	J.-L. Lions and E. Magenes, 
	Non-homogeneous boundary value problems and applications Vol I. 1972.
	\url{https://doi.org/10.1007/978-3-642-65161-8}.
	
	
	\bibitem {MS:12}
	M. Meyries and R. Schnaubelt. 
	Interpolation, embeddings and traces of anisotropic fractional Sobolev spaces with temporal weights.
	\textit{J. Funct. Anal.} \textbf{262} (2012), 1200--1229,
	\url{https://doi.org/10.1016/j.jfa.2011.11.001}.
	
	
	\bibitem {MWW:15}
	A. Mikelic, M. Wheeler and T. Wick.
	Phase-field modeling of a fluid-driven fracture in a poroelastic medium.
	\textit{Comput. Geosci.} \textbf{19} (2015), 1171--1195,
	\url{https://dx.doi.org/10.1007/s10596-015-9532-5}.
	
	
	
	
	
	\bibitem{PS:16}
	J. Prüss and G. Simonett, 
	Moving interfaces and quasilinear parabolic evolution equations. 2016.
	\url{https://doi.org/10.1007/978-3-319-27698-4}.
	
	
	\bibitem {PSW:18}
	J. Prüss, G. Simonett and M. Wilke.
	Critical spaces for quasilinear parabolic evolution equations and applications.
	\textit{J.~Differential Equations} \textbf{264} (2018), 2028--2074,
	\url{https://doi.org/10.1016/j.jde.2017.10.010}.
	
	
	\bibitem {PW:17}
	J. Prüss and M. Wilke.
	Addendum to the paper "On quasilinear parabolic evolution equations in weighted 
	-spaces II".
	\textit{J. Evol. Equ.} \textbf{17} (2017), 1381--1388,
	\url{https://doi.org/10.1007/s00028-017-0382-6}.
	
	
	\bibitem {PW:18}
	J. Prüss and M. Wilke.
	On Critical Spaces for the Navier--Stokes Equations.
	\textit{J. Math. Fluid. Mech.} \textbf{20} (2018), 733--755,
	\url{https://doi.org/10.1007/s00021-017-0342-5}.
	
	
	\bibitem{Saf:71} 
	P. Saffman.
	On the boundary condition at the interface of a porous medium.
	{\it Stud. Appl. Math.} \textbf{1} (1971), 93--101,
	{\url{https://doi.org/10.1002/sapm197150293}}.


	\bibitem{See:72} 
	R. Seeley.
	Interpolation in {$L\sp{p}$} with boundary conditions.
	{\it Studia Math.} \textbf{44} (1972), 47--60.
	\url{https://doi.org/10.4064/sm-44-1-47-60}.
		
	
	\bibitem{Ser:62} 
	J. Serrin.
	On the interior regularity of weak solutions of the Navier-Stokes equations.
	{\it Arch. Mech. Ration. Anal.} \textbf{9} (1962), 187--195.
	\url{https://doi.org/10.1007/BF00253344}.
	

	\bibitem{Sho:00} 
	R. Showalter.
	Diffusion in Poro-elastic Media.
	{\it J. Math. Anal. Appl.} \textbf{251} (2000), 310--340.
	\url{https://dx.doi.org/10.1006/jmaa.2000.7048}.
	
	
	\bibitem{Sho:05} 
	R. Showalter.
	Poroelastic filtration coupled to Stokes flow.
	{\it Control theory of partial differential equations, Lect. Notes Pure Appl. Math.} \textbf{242} (2005), 229--241.
	
	
	\bibitem{Tri:78}
	H.~Triebel,
	{Interpolation Theory, Function Spaces, Differential Operators}. {\it North-Holland}, 1978.
	
	
\end{thebibliography}
\end{document}